\documentclass[11pt]{amsart}
\usepackage{amsfonts,latexsym,amsmath, amssymb}
\usepackage{mathrsfs}
\usepackage{bm}
\usepackage{xcolor}
\usepackage{color}
\usepackage{amsmath,nccmath}
\usepackage{cite}
\usepackage[utf8]{inputenc}
\usepackage{graphicx,enumerate}
\usepackage{enumitem}
\usepackage{epsfig}
\usepackage{graphicx}
\usepackage{graphics}
\usepackage{cite}
\usepackage{enumerate}
\usepackage{caption,subcaption}
\newcommand{\bea}{\begin{eqnarray}}
	\newcommand{\eea}{\end{eqnarray}}
\def\beaa{\begin{eqnarray*}}
	\def\eeaa{\end{eqnarray*}}
\def\ba{\begin{array}}
	\def\ea{\end{array}}
\def\be#1{\begin{equation} \label{#1}}
	\def \eeq{\end{equation}}

\makeatletter
\@namedef{subjclassname@2020}{\textup{2020} Mathematics Subject Classification}
\makeatother
\def\be{{\beta}}

\def\R{{\mathbb{R}}}

\def\Z{{\mathbb{Z}}}

\def \CJ{c^\mathit{cj}}
\def\sgH2{\sigma_H^2}
\def\sgL2{\sigma_L^2}

\newtheorem{theorem}{Theorem}[section]
\newtheorem{lemma}[theorem]{Lemma}
\newtheorem{proposition}[theorem]{Proposition}
\newtheorem{corollary}[theorem]{Corollary}
\newtheorem{definition}[theorem]{Definition}
\newtheorem{remark}[theorem]{Remark}

\newtheorem{claim}[theorem]{Claim}

\newenvironment{claimproof}
{\vspace{-1pc}\proof}
{\endproof}

\numberwithin{equation}{section}

\numberwithin{equation}{section}

\allowdisplaybreaks
\begin{document}
	
	\title[Slow-fast system in Rosales-Majda combustion model]{Slow-fast system in Rosales-Majda combustion model with fractional order kinetics}

	\author[C.-M. Brauner]{Claude-Michel Brauner}
		\address[C.-M. Brauner]{Institut de Math\'ematiques de Bordeaux UMR CNRS 5251, universit\'e de Bordeaux, 33405 Talence, France}
	\email{{\tt claude-michel.brauner@u-bordeaux.fr}}
		\author[J. Jing]{Jinlong Jing}
	\address[J. Jing]{School of Mathematics and Statistics, Yunnan University, Kunming, 615000, China}
	\email{{\tt jingjinlong@stu.ynu.edu.cn}}
		\author[R. Roussarie]{Robert Roussarie}
	\address[R. Roussarie]{IMB UMR 5584, Universit\'e de Bourgogne Europe, CNRS, F-21000 Dijon, France}
	\email{{\tt Robert.Roussarie@u-bourgogne.fr}}


	\begin{abstract}
		We consider traveling wave solutions of a one-dimensional model for detonation waves derived by Rosales and Majda, when the reaction order $\alpha$ belongs to $[0,1)$. The chemical kinetics is a simplified Arrhenius law or a Heaviside function. The model in the reaction zone is a slow-fast dynamical system for a vector representing temperature and mass fraction, which depends on the velocity $c$ and small viscosity $\beta$. Our goal in this paper is to study the bifurcation diagram in the $(\beta,c)$ parameter space and identify the nature of the trajectories corresponding to viscous shock waves. The demonstrations are based on a variety of techniques including the Poincar\'e-Bendixson theorem and the Fenichel theory. Theoretical results are confirmed by numerical computations.
	\end{abstract}
	\maketitle
	
	\centerline{\sl Dedicated to Roger Temam on his 85th Birthday, with Respect and Admiration.}
	
	\section{Introduction}
	
	In this paper, we consider the one-dimensional model derived by Rosales and Majda (see \cite{RM83}) for studying detonation waves when the wave velocity approaches the speed of sound, in the wake of the famous ``Majda model'' introduced in 1981 by Andrew Majda (see \cite{M81,CMR,MR}). More specifically, according to \cite{R92,Li92}, we will study a range of models where the order of the chemical reaction, denoted by $\alpha$, lies between $0$ and $1$ (see \cite{BRSZ21} for a similar problem). The original model proposed by Rosales and Majda corresponds to $\alpha=1$, see \cite{RM83,S99}.
	
	In dimensionless variables, the Rosales--Majda model with fractional reaction order $0\leqslant\alpha<1$ is expressed as follows:
	\begin{equation}\label{fractional_RM}
		\begin{cases} 
			T_t + \left(\frac{T^2}{2} - q_0 Z\right)_x =\beta T_{xx},\\[1.5mm]
			Z_x=K\varphi(T)Z^{\alpha},\\
			\lim_{x \to +\infty}Z(x,t)= 1.
		\end{cases} 
	\end{equation}
	In the system \eqref{fractional_RM}, $x\in\R$ and $t>0$ are the space and time variables, respectively, $T(x,t)$ is a lumped variable representing the temperature, $Z(x,t)$ is the mass fraction such that $0\leqslant Z(x,t)\leqslant 1$. Furthermore, $q_0>0$ is the heat release, $\beta>0$ is the (small) viscosity and $K>0$ is the reaction rate, $\varphi(T)$ stands for the chemical kinetics.
	
	Our aim is to study the existence of traveling wave solutions  $T(\xi)$, $\xi=x-ct$, of system \eqref{fractional_RM} 
	propagating at a velocity $c>0$ and connecting two equilibria $T_{-}$ and $T_{+}$ such that $T_{+}<T_i<T_{-}$ where $T_i$ is the ignition temperature. The structure of a solution is a \textit{traveling reactive shock wave}, namely a viscous shock followed by a reaction zone that ends when $Z$ vanishes. More precisely, assuming for convenience that $T$ reaches the ignition temperature at $\xi=0$, the real line will be divided into the shock wave region $\{\xi>0\}$, the reaction zone $\{-\ell<\xi<0\}$, and beyond, the region $\{\xi<-\ell\}$. Here, $\xi=-\ell$ is a finite \textit{free interface} (see \cite{BRSZ21,Li92}) defined by $Z(-\ell)=0$. Obviously, the free interface is at~$-\infty$ when $\alpha=1$ as in \cite{RM83,S99}.
	
	In this paper, in addition to the free interface, there are other significant differences from \cite{RM83,S99}. We will reformulate the problem in the reaction region as a slow-fast system for small viscosity $\beta$, see Section \ref{slowfast_intro}. Moreover, in contrast to \cite{RM83,S99}, the heat release $q_0$ is a fixed parameter. Instead, we will search for a bifurcation diagram in the parameter space $(\beta,c)$, see Figure \ref{diagram_intro}.

	\subsection{Formulation of the problem}\label{model}
	
	We now take a closer look at the model \eqref{fractional_RM}. According to \cite{RM83}, the chemical kinetics is an ignition temperature kinetics of the Arrhenius type. In the following we assume that the ignition temperature $T_i$ is a real number such that (see \cite{Li92})
	\begin{align}\label{condition_Ti}
		0<T_i<\sqrt{2q_0}.
	\end{align}
	For simplicity, we assume also that $\varphi$ is discontinuous at the ignition temperature $T_i$, 
	$\varphi(T)=0$ for $T\in [0,T_i)$, and is sufficiently smooth and positive whenever for $T\geqslant T_i$. There will be an additional technical hypothesis for large $T$, see below.  In this paper, the following two examples of functions $\varphi$ are considered:
	\begin{enumerate}[label=(\roman*), itemsep=1ex]
		\item the Heaviside $\varphi(T)= H(T-T_i)$; 
		\item a simplified form of the Arrhenius law that reads (see also \cite{HLZ13,HHLZ15})
		\begin{align}
			\varphi(T)= H(T-T_i) \exp\left(-\frac{T_a}{T}\right),
		\end{align}
	\end{enumerate}
	where $T_a=E_a/R>0$, where $E_a$ is the activation energy and $R$ is the gas constant (see, e.g., \cite[Section 1.3]{BL}). Obviously, case (i) corresponds to the limit case $T_a=0$, so the Heaviside function can be regarded as a simplification of the Arrhenius law.

	For given $K, q_0>0$, $0\leqslant\alpha<1$, we seek a traveling wave solution $(T(x-ct),Z(x-ct))$, propagating from the left to right with velocity $c>0$ to be determined. In the coordinate $\xi=x-ct$,
	%\begin{align*}
	%\xi=x-ct,
	%\end{align*}
	we search for $T(\xi)$ and $Z(\xi)$,  continuously differentiable and continuous functions, respectively, solutions of
	\begin{equation}\label{TW_2d_order}
		\begin{cases} 
			-cT_{\xi} + \left(\frac{T^2}{2} - q_0 Z\right)_{\xi} =\beta T_{\xi\xi}, \\[2mm]
			Z_{\xi}=K\varphi(T)Z^{\alpha}.
		\end{cases} 
	\end{equation}
	with conditions at infinity as follows:
	\begin{equation}\label{asymptotics_TW}
		\begin{cases} 
			\lim_{\xi \to +\infty}T(\xi)= T_{+}, \quad \lim_{\xi \to +\infty}Z(\xi)= 1,\\[1mm]
			\lim_{\xi \to -\infty}T(\xi)= T_{-},
		\end{cases} 
	\end{equation}
	where $T_{-}$ and $T_{+}$ are such that $T_{+}<T_i<T_{-}$.
	
	Integration of \eqref{TW_2d_order} yields 
	\begin{equation}\label{TW_1st_order}
		\begin{cases} 
			\beta T_{\xi}= -cT +\frac{T^2}{2} - q_0 Z +A, \\[2mm]
			Z_{\xi}=K\varphi(T)Z^{\alpha}.
		\end{cases} 
	\end{equation}
	where $A$ is a constant. Throughout this paper, for simplicity, we choose $T_+=0$; hence, $A=q_0$ and the differential system reads as follows:
	\begin{equation}\label{TW_1st_order}
		\begin{cases} 
			\beta T_{\xi}= -cT +\frac{T^2}{2} + q_0(1-Z), \\[2mm]
			Z_{\xi}=K\varphi(T)Z^{\alpha},
		\end{cases} 
	\end{equation}
	with conditions \eqref{asymptotics_TW} at infinity. 
	
	Moreover, taking advantage of the translation invariance, we fix
	\begin{align}\label{translation_inv}
		T(0)=T_i,
	\end{align}
	and seek a solution which verifies
	\begin{align}\label{}
		T(\xi)<T_i\;\, {\rm when}\;\, \xi>0, \qquad 	T(\xi)>T_i \;\, {\rm when}\;\, \xi<0,
	\end{align}
	hence, $T_{-}= \lim_{\xi \to -\infty}T(\xi)>T_i$.
	\vskip 2mm
	As we mentioned above, there are the following three regions, keeping in mind that $T \in C^1(\R), Z \in C^0(\R)$: 
	\begin{enumerate}[label=(\roman*), wide, labelwidth=!, labelindent=0pt]
		\item The shock wave region $\{\xi>0\}$: because $T(\xi)<T_i$, it holds $\varphi(\xi)=0$, therefore $Z(\xi)=1$ whenever $\xi>0$. The system \eqref{TW_1st_order} for $(T,Z)$ reads simply in the shock wave region:
		\begin{equation}\label{shock_system}
			\begin{cases} 
				\beta T_{\xi}=-cT+ \frac{T^2}{2}, \\[1mm]
				Z\equiv 1,\\[1mm]
				T(0)=T_i, \quad \lim_{\xi \to +\infty}T(\xi)=0.
			\end{cases} 
		\end{equation}
		
		\item The reaction zone $\{-\ell<\xi<0\}$. The quantity $-\ell,\, \ell>0$, is a free interface that we may also call the \textit{trailing interface} by analogy with \cite{BGKS15,BRSZ21}, defined by
		\begin{align}
			-\ell = \inf\left\{\xi<0, Z(\xi)>0\right\}.
		\end{align}
		Consider the case where $\varphi(T)= H(T-T_i)$: because $T>T_i$ it holds $\varphi(T)=1$ when $\xi<0$, hence $Z_{\xi}=KZ^{\alpha}$ for $\xi<0$ with $Z(0)=1$. Then, 
		\begin{align*}
			\ell=\frac{1}{K(1-\alpha)}, \quad 0\leqslant\alpha<1,
		\end{align*}
		that is, $\ell<+\infty$ (see \cite{Li92}). 
		
		Therefore, we anticipate that, whenever $0\leqslant\alpha<1$, the reaction zone is finite and the system \eqref{TW_1st_order} for $(T,Z)$ reads
		\begin{equation}\label{reaction_system}
			\begin{cases}
				\beta T_{\xi}= -cT+ \frac{T^2}{2} +q_0(1-Z), \\[1mm]
				Z_{\xi}= K\varphi(T)Z^{\alpha},\\[1mm]
				T(0)=T_i, \quad Z(0)=1,
			\end{cases} 
		\end{equation} 
		and the condition $Z(-\ell)=0$ which determines $\ell$ and provides the value of $T(-\ell)$. 
		\vskip 2mm
		\item In the region $\{\xi<-\ell\}$ following the finite reaction zone, the system \eqref{TW_1st_order} for $(T,Z)$ reads:
		\begin{equation}\label{beyond_reaction}
			\begin{cases}
				\beta T_{\xi}= -cT+ \frac{T^2}{2} +q_0, \\[1mm]
				Z\equiv 0, \\[1mm]
				\lim_{\xi \to -\infty}T(\xi)=T_{-}.
			\end{cases} 
		\end{equation}
		The candidates for $T_{-}$ are $T_0(c)=c+\sqrt{c^2-2q_0}$ and $T_1(c)=c-\sqrt{c^2-2q_0}$ with the condition $c\geqslant\sqrt{2q_0}$. By analogy with detonations (see \cite{RM83}), the equilibria $(T_0(c),0)$ and $(T_1(c),0)$ correspond, respectively, to a \textit{strong reactive wave} and a \textit{weak reactive wave}.
		
		\begin{definition}\label{CJ}
			Again by analogy, the minimal value $c =\sqrt{2q_0}$
			is called the Chapman-Jouguet (CJ) velocity and will be denoted by $\CJ$ throughout the paper.
		\end{definition}
		We observe that $T_0(c)\geqslant c\geqslant \CJ >T_i$ (see \eqref{condition_Ti}) and also $T_1(c)>0.$ 
		For $c\rightarrow +\infty$  we have that $T_0(c)\rightarrow +\infty$ and $T_1(c)\rightarrow 0_+$ (more precisely, $T_1(c)\sim {q_0}/{c^2}$).
	\end{enumerate}

	\subsection{Background results}\label{main_section}
	
	In contrast with \cite{RM83,S99}, we consider system \eqref{shock_system}--\eqref{beyond_reaction} at fixed $q_0>0$ and look for traveling wave solutions $(T(\xi),Z(\xi)) \in C^1(\R)\times C^0(\R)$ which travel at speed $c\geqslant \CJ$. Trajectories which connect equilibria $(T_0(c),0)$, $(T_1(c),0)$, $(\sqrt{2q_0},0)$ to the equilibrium state $(T_{+},1)$ correspond, respectively, to strong reactive solutions, weak reactive solutions, Chapman-Jouguet reactive solutions.
	
	More specifically, when $c>\CJ$, there are three types of strong reactive traveling waves such that $T_{-}=T_0(c)$: 
	\begin{enumerate}[label=(\roman*)]
		\item  monotonic solutions such that $T(\xi) \to T_0(c)$ exponentially when $\xi \to -\infty$;
		\item  non-monotonic solutions with a bump (a combustion spike), in the reaction zone such that $T(\xi) \to T_0(c)$ exponentially when $\xi \to -\infty$;
		\item  special solutions, or solutions of special type, such that $T\equiv T_0(c)$ beyond the finite reaction zone. We may also call them \textit{finite time solutions} in the context of the slow-fast system and the ``time'' $\tau$. 
	\end{enumerate}
	
	There are also weak reactive waves of special type such that $T_{-}=T_1(c)$ and $T\equiv T_1(c)$ beyond the finite reaction zone.
	
	In the case $c=\CJ$, then $T_{-}=\sqrt{2q_0}$ and there are only two types of solutions, namely Chapman-Jouguet (CJ) solutions with a bump such that  $T(\xi) \to \sqrt{2q_0}$ algebraically when $\xi \to -\infty$, and solutions of special type such that $T\equiv \sqrt{2q_0}$ beyond the finite reaction zone.
	
	%There may also be weak detonation waves such that $T(\xi) \to T_{-}=T_1(c)$ exponentially when $\xi \to -\infty$. 
	\vskip 2mm
	We summarize our results in the following theorems,  with the notation
	\begin{align}\label{c_star}
		c_*= \frac{T_i^2+2q_0}{2T_i}.
	\end{align}
	As $T_i>0$ and $q_0>0,$ we have that $c_* > c^{c_j}.$
	In the domain $(\beta>0,c \geqslant \CJ)$, it holds (see Figure \ref{diagram_intro}):
	\begin{theorem}\label{theorem_intro1}
		(i) There exists a function  $\beta_0(c) :[\CJ,+\infty)\rightarrow \R^+,$ 
		of class $\big[\frac{1}{1-\alpha}\big],$ such that there is a strong reactive solution of special type for $\beta=\beta_0(c)$;\\
		(ii) there exists a function 
		$\beta_1(c) :\big[\CJ,c_{\ast}\big)\rightarrow \R^+,$ of class $\big[\frac{1}{1-\alpha}\big],$ such that there is a weak reactive solution of special type for $\beta=\beta_1(c);$\\
		(iii) the curves $\beta_0(c)$ and $\beta_{1}(c)$ enjoy the following properties:
		\begin{enumerate}[label=(\alph*)]
			\item for  all $c>\CJ$, we have that $\beta_0(c)<\beta_1(c);$ 
			\item $\beta_0(c)\rightarrow +\infty$ as $c\rightarrow +\infty$;
			\item $\beta_1'(c)>0$ for $c\in \big(\CJ,c_{\ast}\big),$ moreover $\beta_1(c)\rightarrow +\infty$ when $c$ $\to c_{\ast}$;
		\end{enumerate}	
		(iv) for $c>\CJ, 0<\beta<\beta_0(c)$, there is a strong reactive solution with a bump; for $\beta_0(c)<\beta<\beta_1(c)$, there is a strong reactive wave of monotonic type; finally, when $\CJ<c<c_{\ast}$ and $\beta \geqslant \beta_1(c)$, there is no solution. 
	\end{theorem}

	An important result is that the bifurcation diagram admits (at least) two critical points. The first one belongs to the axis $c=\CJ$.
	
	\begin{theorem}\label{prop-global_intro2}{\rm(Critical point $(\beta^{\mathit{cj}}, \CJ)$)}. It holds:
		\begin{enumerate}[label=(\roman*)]
			\item
			for $c=\CJ$, there exists a unique $\beta^{\mathit{cj}}\in (0,+\infty)$ such that $\beta^{\mathit{cj}}=\beta_0(\CJ)=\beta_1(\CJ)$;
			\item 
			at the critical point $(\beta^{\mathit{cj}}, \CJ)$, there is a Chapman-Jouguet reactive wave of special type.
			\item
			the union of the graphs of $\beta_0$ and $\beta_1$ with the point $(\beta^{\mathit{cj}},\CJ)$ is a simple curve $(B)$ in the parameter space, of class $\big[\frac{1}{1-\alpha}\big].$ 
			The mapping $\tilde{c}$: $\beta\to c$ defined by the curve $(B)$ locally around $(\beta^{\mathit{cj}},\CJ)$ has a quadratic minimum at $\beta^{\mathit{cj}}$, i.e., ${\tilde{c}}'(\beta^{\mathit{cj}})=0$, ${\tilde{c}}''(\beta^{\mathit{cj}})>0$;
			\item for $c=\CJ$, $\beta\in(0,\beta^{cj})$, there is a Chapman-Jouguet reactive wave with a bump. 
		\end{enumerate}
	\end{theorem}
	The second critical point is a turning point in the domain $(\beta>0,c>\CJ)$.
	\begin{theorem}\label{theorem_intro3}{\rm{(Turning point $(\bar{\beta},\bar{c})$)}}
		There exists a turning point $(\bar{\beta},\bar{c})$ which verifies the following properties:
		\begin{enumerate}[label=(\roman*)]
			\item
			$\beta_{0}(\bar{c})=\bar{\beta}$, $\beta_{0}'(\bar{c})=0$;
			\item 
			$0<\bar{\beta}\leqslant\beta_{0}(c)$ for all $c\geqslant c^{cj}$;
			\item at the turning point $(\bar{\beta},\bar{c})$, there is a strong reactive wave of special type;
			\item for $\beta\in(0,\bar{\beta})$ and  $c>\CJ$, there is a strong reactive wave with a bump.  
		\end{enumerate}	
	\end{theorem}
	
	A significant consequence is that:
	\begin{corollary}
		There is a continuum of admissible speed $\{c\geqslant\CJ\}$ whenever $\beta\in(0,\bar{\beta})$.
	\end{corollary}
	
	\begin{figure}[htp]
		\begin{center}
			\includegraphics[scale=0.9]{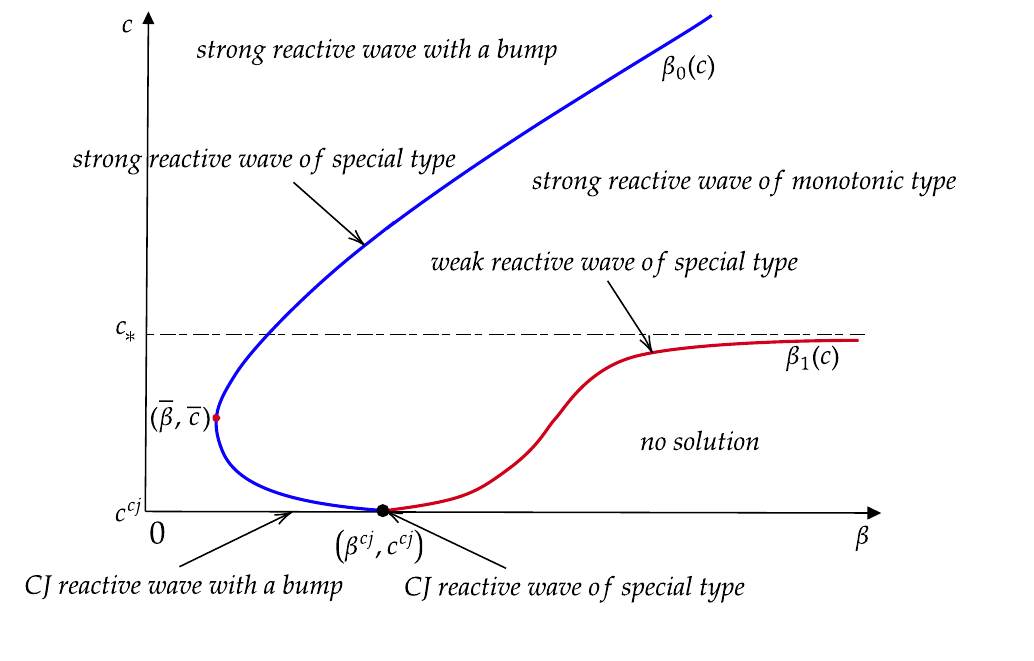}\\
			\caption{This figure provides a qualitative representation of the bifurcation diagram. Note that the only unproven point in this paper concerns the monotonicity of the branch of reactive waves of special type. Thus far, it has only been proven that this branch is locally the graph of $\beta$ as a function of $c$.}
			\label{diagram_intro}
		\end{center}
	\end{figure}
	
	\begin{remark}\label{rem-1}
		In Figure \ref{diagram_intro} there is a unique turning point for the function $\beta_0(c)$.  But this is only numerical evidence that remains to be proven.
	\end{remark}
	
	\subsection{Organization of the paper}\label{slowfast_intro}
	We focus on the region $\{\xi<0\}$ and the vector field
	\begin{equation}\label{vector_fiels}
		X :=
		\left\{
		\begin{array}{l}
			T_{\xi} = \frac{1}{\beta}\left\{-cT+\frac{1}{2}T^2+q_0(1-Z)\right\},  \\[2mm]
			Z_{\xi} = K\varphi(T) Z^\alpha.
		\end{array}
		\right.
	\end{equation}
	that we shall consider as a slow-fast system for $\beta\rightarrow 0$. As in \cite{RM83,S99}, we introduce
	the \textit{fast time} 
	\begin{align}\label{fast}
		\tau =\frac{\xi}{\beta}; 
	\end{align}
	Denoting by $\dot{f}$ the differentiation in $\tau$, the field $X$ turns into 
	\begin{equation}\label{eq-tau}
		\tilde X :=
		\left\{
		\begin{array}{l}
			{\dot T} = -cT+\frac{1}{2}T^2+q_0(1-Z),  \\[2mm]
			{\dot Z} =\beta K\varphi(T) Z^\alpha.
		\end{array}
		\right.
	\end{equation}
	First of all, because we are interested in small viscosity, unlike \cite {RM83,S99} we perform a straightforward change of variable in \eqref{eq-tau} that maintains $\beta$ in the system. We simply set:
	\begin{align}
		\widetilde{\beta}=\beta K,
	\end{align}  
	and hereafter omit the tilde. Therefore, we will consider the following slow-fast system on $Q=\{(T,Z)\in \R^2\  |\  T \geqslant T_i, Z\geqslant 0\}$:
	\begin{equation}\label{eq-tau-bis}
		\tilde X :=
		\left\{
		\begin{array}{l}
			{\dot T} = -cT+\frac{1}{2}T^2+q_0(1-Z),  \\[2mm]
			{\dot Z} =\beta \varphi(T)Z^\alpha.
		\end{array}
		\right.
	\end{equation}
	We write $P_c(T)=-cT+\frac{1}{2}T^2+q_0.$
	The roots of $P_c$ are $T_0(c)=c+\sqrt{c^2-2q_0}$ and $T_1(c)=c-\sqrt{c^2-2q_0}$ (see above), 
	the  zeros of $\tilde X$ on $Q$ are $(T_0(c),0)$ and $(T_1(c),0)$. Also,  $\tilde X$ has a quadratic contact with the vertical line $\{T=T_i\}$ at the point $(T_i,Z_i(c))$ where $Z_i(c)=\frac{1}{q_0}P_c(T_i).$   The point $(T_i,Z_i(c))$ belongs to $Q$ only if $T_1(c)\geqslant T_i$ which is the condition $c\leqslant c_*$ , see \eqref{c_star}.
	
	Section \ref{phase_portraits} is devoted to the study of the phase portraits. At the outset, we introduce the new variable $U=\frac{1}{1-\alpha}Z^{1-\alpha}$. Making the change of coordinates
	$\Phi := (T,Z)\rightarrow (T,U)$, the field $\tilde X$  becomes $\tilde X^U=\Phi_*(\tilde X)$
	\begin{equation}\label{intro_eq-tildeXU}
		\tilde X^U := 
		\left\{
		\begin{array}{l}
			{\dot T} = -cT+\frac{1}{2}T^2+q_0\Big(1-(1-\alpha)^{\frac{1}{1-\alpha}}U^{\frac{1}{1-\alpha}}\Big),  \\
			{\dot U} =\beta \varphi(T),
		\end{array}
		\right.
	\end{equation}
	the big advantage being that the vector field $\tilde X^U$ has no zero on $Q$.  There are two particular orbits which are crucial for the study of the problem: the orbit $\gamma_1(\beta,c)$ starting at the point $(T_1(c),0)$  and arriving at a point $(T_i,U_1(\beta,c))$  and the orbit $\gamma_0(\beta,c)$ starting at the point $(T_0(c),0)$  and arriving at a point $(T_i,U_0(\beta,c)).$ Then, we deduce the phase portrait of $\tilde X$  from the one of $\tilde X^U$ by taking its image by the map $\Phi^{-1}.$ The case of the Chapman Jouguet velocity $\CJ$ is the subject of Subsection \ref{case_CJ}.
	
	In Section \ref{existence}, we study the nature of the trajectory in negative times of $\tilde X$,  passing through $(T_i,1 )$   in function of the parameter $(\beta,c)$. More specifically, we want to determine if this trajectory is a solution and identify its type. We first consider in Subsection \ref{limite_cases} the limit cases at fixed $c\geqslant \CJ$, namely $\beta \to 0$ and $\beta \to \infty$. Without resorting to the technicalities of the theory of slow-fast systems, we are satisfied with a basic proof: the orbit through the point $(T_i,1)$ is a solution with a bump. For $\beta=+\infty$, there is a vertical vector field $\tilde X^U_{\infty,c}$ without zeros. In Subsection \ref{rotating_property}, we prove a {\it rotating property} (see \cite{DU}) for the vector field  $\tilde X^U_{\beta,c}$, in function of the parameters $\beta$ and $c$.
	
	Section \ref{bifurcation_diagrams} is devoted to the bifurcation diagram in the $(\beta,c)$ space and to the rigorous proofs of Theorems \ref{theorem_intro1}, \ref{prop-global_intro2} and \ref{theorem_intro3} which can be inferred from the previous sections.
	The bifurcation curves $\beta=\beta_1(c)$ and
	$\beta=\beta_0(c)$ are simple curves for $c>\CJ$.  In a small neighborhood of the point $(\beta^{\mathit{cj}},\CJ)$, the union of the graphs of $\beta_0(c)$ and $\beta_1(c)$ is also a simple curve, to be denoted by $(B)$. 
	
	The behavior as $c$ and $\beta$ tend to $+\infty$ is trickier and is the subject of Subsection \ref{large}. As little is known about the behavior of the function 
	$\beta_0(c)$ in the large, we introduce another technical hypothesis $(H_\infty)$ about the function $\varphi$ for large $T$, which is verified by the two examples: the function $\psi(V)=\varphi(\frac{1}{V})$ is smooth at $V=0$ with a value $\psi(0)>0$. We prove that  $\beta_0(c)\rightarrow +\infty$ if $c\rightarrow +\infty$, see Proposition \ref{prop-c-infty}.
	The lengthy proof is divided in several steps. In particular, we use the Fenichel theory in the  specific case of two-dimensional smooth slow-fast systems which shows existence of center manifolds  (see \cite{DDMR}), and compute an expansion of a center manifold explicitly.
	
	Finally, to confirm our theoretical results, we performed numerical computations for the bifurcation diagram in Section \ref{numerics}, starting from the vector field:
	\begin{equation}\label{eq-tildeXU_bis}
		\tilde X^U := 
		\left\{
		\begin{array}{l}
			{\dot T} = -cT+\frac{1}{2}T^2+q_0\Big(1-(1-\alpha)^{\frac{1}{1-\alpha}}U^{\frac{1}{1-\alpha}}\Big),  \\
			{\dot U} =\beta\varphi(T).
		\end{array}
		\right.
	\end{equation}
	Two methods are implemented: an analytical method for the Heaviside function in the special case $\alpha = 0.5$ (see Appendix \ref{appendixA}) and a shooting method for the Arrhenius law (see Appendix \ref{AppendixB}). 
	In particular, our numerical results show an almost complete overlap of the two curves produced by the two methods for the Heaviside function. It is also worth noting a consistent behavior of the bifurcation curves $\beta_0(c)$ and $\beta_1(c)$ across different modeling scenarios (Heaviside and Arrhenius). This similarity suggests that the underlying dynamics of the system is robust to variations in the functional form of the reaction rate.  Finally, we note that numerically we have observed a single turning point $(\bar{\beta},\bar{c})$. In future  research, it  is  expected  that  the  global  uniqueness  of  the turning point will be proved.

	\section{Phase portraits}\label{phase_portraits}
	
	We consider the vector field
	\begin{equation}\label{eq-tau_bis}
		\tilde X :=
		\left\{
		\begin{array}{l}
			{\dot T} = -cT+\frac{1}{2}T^2+q_0(1-Z),  \\[2mm]
			{\dot Z} =\beta \varphi(T) Z^\alpha
		\end{array}
		\right.
	\end{equation}
	on $Q=\{(T,Z)\in \R^2\  |\  T \geqslant T_i, Z\geqslant 0\}$ and write $P_c(T)=-cT+\frac{1}{2}T^2+q_0.$
	We introduce the homocline
	$h_c := \{Z=\frac{1}{q_0}P_c(T)\}.$ Along this curve the field is vertical. This line will be used in order to obtain some differentiable properties of the flow of $\tilde X.$  
	
	For convenience, we recall $c_*= \frac{T_i^2+2q_0}{2T_i}$ (see \eqref{c_star}).
	If $c<c_*$, $h_c$ has two connected components $h_c^l$ and $h_c^r$ which are two arcs of parabola : $h_c^l$ from $(T_1(c),0)$ to $(T_i,Z_i(c))$ and $h_c^r$ going from $(T_0(c),0)$ to infinity.
	If $c>c_*$ we have just one connected component $h_c^r.$ 
	
	In order to simplify the arguments  about the phase portrait of $\tilde X$, in Section \ref{tilde_U}  we introduce the new variable 
	$U=\frac{1}{1-\alpha}Z^{1-\alpha}$ and make the change of coordinates:
	$$\Phi := (T,Z)\rightarrow (T,U).$$
	
	If $c<c_*$,  $Q\setminus h_c$ has three connected components: the central region $C_c$ above $h_c$, the left region $L_c$ below $h_c^l$ and the right region $R_c$ below $h_c^r.$ For $c\geqslant c_*$, we have just the two regions $C_c$ and $R_c$.
	See Figure \ref{fig-field-X-U} in $(T,U)$-coordinates or Figure \ref{fig-field-X-Z} in $(T,Z)$-coordinates.
	The boundary of $C_c$ is equal to $\partial C_c=h_c\cup [T_1(c),T_0(c)]\times \{0\}.$

	\subsection{Phase portrait of $\tilde X^U$}\label{tilde_U}
	We  introduce the new variable 
	$U=\frac{1}{1-\alpha}Z^{1-\alpha}$ and make the change of coordinates :
	$$\Phi := (T,Z)\rightarrow (T,U).$$
	In coordinates $(T,U)$, the field $\tilde X$  becomes $\tilde X^U=\Phi_*(\tilde X)$ given by:
	
	\begin{equation}\label{eq-tildeXU}
		\tilde X^U := 
		\left\{
		\begin{array}{l}
			{\dot T} = -cT+\frac{1}{2}T^2+q_0\Big(1-(1-\alpha)^{\frac{1}{1-\alpha}}U^{\frac{1}{1-\alpha}}\Big),  \\[0.5mm]
			{\dot U} =\beta \varphi(T).
		\end{array}
		\right.
	\end{equation}
	Note that $\tilde X^U$ has  no  more zero, in fact it has a non-zero vertical component and in particular is transverse to the axis $\{U=0\}$). Moreover, it is more differentiable than $\tilde X$, which results from the fact that the change of coordinates is itself not differentiable along $\{Z=0\}.$ The non-differentiability is translated into the first line, but  $\tilde X^U$  is at least of class $\big[\frac{1}{1-\alpha}\big]\geqslant 1.$ 
	Then,  we can apply the  Cauchy theorem to $\tilde X^U$ on the whole closed domain
	$Q.$ This implies the {\it uniqueness of trajectories} through  each point of $Q$. We shall show that this is not the case for $\tilde X.$ 
	
	The vertical lines are preserved by the mapping $\Phi := (T,Z)\mapsto (T,U)$ which is a smooth diffeomorphism on  $Q^+=\{T\geqslant T_i, \  Z>0\}.$ All the differentiable properties verified on $Q^+$ are preserved, for instance  $X^U$ has also a quadratic contact with $\{T=T_i\}$ at 
	$(T_i,U_i(c))$  where  $(1-\alpha)^{\frac{1}{1-\alpha}}U_i(c)^{\frac{1}{1-\alpha}}=\frac{1}{q_0}P_c(T_i).$  Also the homocline of $\tilde X$ is sent on the homocline of $\tilde X^U.$
	We keep the same names as before for the homocline, its connected components and the connected components of $Q\setminus h_c.$   
	The homocline $h_c$  is now given by $\{(1-\alpha)^{\frac{1}{1-\alpha}}U^{\frac{1}{1-\alpha}}=\frac{1}{q_0}P_c(T)\}.$  
	At the points $(T_0(c),0)$ and $(T_1(c),0)$ this homocline has a contact of order $\frac{1}{1-\alpha}$ with the vertical, which is the direction of $\tilde X^U$ at these points. Elsewhere, $h_c$ remains transverse to the vertical direction.
	\par
	As the vector field $\tilde X^U$ has no zero on $Q$, it is easy to describe its phase portrait on this domain, using the qualitative ideas of Poincar\'e (see \cite{P,B}; more recent references include \cite{L, DLA, PDM}). Nevertheless, minor  difficulties arise from the non-compactness of $Q$.  We proceed by giving a series of claims: 
	
	\begin{claim}\label{claim(i)} Each orbit $\gamma$ (curve image of trajectory) is a graph 
		$U\mapsto\tilde T_\gamma(U)$ above some interval in $U$.
	\end{claim}
	\begin{claim}\label{claim(ii)}
		Each orbit $\gamma$ has   at most    one  intersection point 
		$(T_\gamma,U_\gamma,)$  with $\partial C_c.$ When 
		$U_\gamma>0,$ i.e. when  $(T_\gamma,U_\gamma,)\not\in [T_1(c),T_0(c)]\times \{0\},$
		the function $\tilde T_\gamma(U)$ has a maximum at $U_\gamma$. 
	\end{claim}
	\begin{claimproof} The first claim  comes from  the fact  that $\tilde X^u$ has a non-zero vertical direction and then each orbit $\gamma$ is a graph 
		$U\mapsto \tilde T_\gamma(U)$ above some interval in $U$.  Next,  as   $\tilde X^u$   points upward along $\partial C_c$ each orbit $\gamma$  cuts $\partial C_c$ in at most one point. At a point
		$(T_\gamma, U_\gamma)$ of $\gamma$ where $\tilde X^U$ is vertical, the function 
		$\tilde T_\gamma$ has a zero-derivative at $U_\gamma.$ Along $\gamma,$ the vector field $\tilde X^U$ has a positive slope if $U<U_\gamma$ and a negative slope if $U>U_\gamma:$ the function $\tilde T_\gamma$ has a maximum at $U_\gamma.$ 
	\end{claimproof}
	
	\begin{claim}\label{claim(iii)} Let be $U_0>U_i(c).$ The trajectory through the point $(T_i,U_0)$ arrives in a finite negative time on    $\partial C_c=h_c\cup[T_1(c),T_0(c)]\times \{0\}$.
	\end{claim} 
	\begin{claimproof}
		We look at the  trajectory $\varphi(\tau)$  of $\tilde X^U$ through the point $(T_i,U_0),$ in negative times.  This trajectory remains in the strip $S=\{0\leqslant U\leqslant U_0\}$. As $B=S\cap C_c$  is  compact, it follows from the Poincar\'e-Bendixson theorem (see, e.g., \cite{P,B}) that   
		$\varphi(\tau)$ arrives in a finite  negative time $\tau_0$ at a point $m_0$ in  boundary $\partial B$ of $B.$ This time $\tau_0$ is strictly negative because $U_0>U_i(c)$ and then  the $U$-coordinate of $m_0$ is strictly smaller than $U_0$ (the coordinate $U$ is strictly discreasing along the trajectory). Then $m_0$ belongs to the part of $\partial B$ contained in $\partial C_c$.$\square$
	\end{claimproof}
	
	\begin{claim}\label{claim(iv)} Conversely, the   trajectory of $X^U$ by any point of  $\partial C_c$  is contained in $C_c$: it goes upward and toward the left and arrives on $\{T_i\}\times [Ui(c),+\infty)$ in a finite positive time.
	\end{claim}
	\begin{claimproof}
		Take any point $m\in \partial C_c.$ As any trajectory must enter into $C_r$ in positive times (for $m=T_1(c)$ or $T_0(c)$, it is convenient to use a tubular neighborhood), and cuts at most one time $\partial C_c$ by Claim \ref{claim(i)}, the (positive) trajectory of $m$ is entirely contained in $C_c.$ It is clear that a positive trajectory in $C_c$ goes upward and toward the left. 
		
		It remains to prove that this  trajectory  arrives on $\{T_i\}\times [Ui(c),+\infty)$ in a finite positive time. We proceed in two steps. First we choose an $\bar{U}>U_i(c)$ also greater than the $U$-coordinate of $m.$
		We call now $M$ the  compact part of $C_c$ below  the horizontal line $\{U=\bar{U}\}.$ Using the Poincar\'e-Bendixson theorem we have that the trajectory through  $m$ arrives in a finite positive time at a point  in 
		$\{T_i\}\times [U_i(c),\bar{U}]$ or in  $\{U=\bar{U}\}$. In the first case  we have finished. In the second case, we arrive at a point $p'=(\bar{T},\bar{U})\in C_c,$  with $\bar{T}>T_i.$
		
		We proceed to the second step of the proof,  considering the trajectory from $p'$, in the positive time.  This trajectory remains in the 
		vertical strip $S=\{U\geqslant \bar{U}, T\leqslant \bar{T}\}$ and it is clear that there exists a $A>0$ such that  the horizontal component of $\tilde X^U$ is less than $-A$ on $S$ (on the segment $\{ (T,\bar{U} | T\in [T_i,\bar{T}]\}$, this horizontal component is strictly negative and then less than $-A$ for some $A>0,$ by compactness; moreover this component is decreasing when $U$ increases). On the other side, the vertical component of $\tilde X^U$ is positive and {\it less than $\beta B$}, where $B={\rm Max}\{\varphi(T) | T\in [T_i,\bar{T}]\}>0.$
		%greater than $\beta B$ with  $B={\rm Inf} \{\varphi(T) | T\in [T_i,\bar{T}]\}.$ 
		 %It follows from the assumptions on %$\varphi(T)$ that $B>0.$  
		 From these two estimations it follows trivially that  the trajectory through $p'$ arrives on $\{T=T_i\}$ in a finite positive time.
	\end{claimproof}
	
	\begin{claim} Any trajectory starting at a point of $h_r\setminus \{(T_0(c),0)\}$ arrives in a finite negative time at a point $(T,0)$ with $T>T_0(c)$.
	\end{claim} 
	
	\begin{claimproof} As such a trajectory remains in a compact subset of $R_c$, we can again apply the Poincar\'e-Bendixson theorem to prove it arrives in a finite negative $-\tau_0<0$  at a point $(T,0)$ with $T\geqslant T_0(c)$. Of course we have that $T>T_0(c)$ because if $T=T_0(c)$ the trajectory would  be, as consequence of (4),  in the interior of $C_c$ for any time $-\tau,$ with $-\tau_0<-\tau<0,$ that  is impossible. 
	\end{claimproof}

	It is now very easy to draw the phase portrait of $\tilde X^U$, see Figure \ref {fig-field-X-U}.
	It is important to notice that  the vector field is transverse to the horizontal axis 
	$\{U=0\},$ tilted to the left along the interval $]T_1(c),T_0(c)[$  (i.e. has a negative slope) and tilted to the right outside (i.e. has a positive slope). There are two particular orbits : the orbit $\gamma_1(\beta,c)$ starting at the point $(T_1(c),0)$  and arriving at a point $(T_i,U_1(\beta,c))$  and the orbit $\gamma_0(\beta,c)$ starting at the point $(T_0(c),0)$  and arriving at a point $(T_i,U_0(\beta,c)).$ These orbits are crucial for proving the results of Section 3.  At $(T_0(c),0)$ and $(T_1(c),0)$ these orbits have a vertical tangent, but it is possible to prove more by a direct computation:
	
	\begin{lemma}\label{lem-transverse}
		Let be $i=0$ or $1$. There are  constants $k_i>0,$  depending on the  parameter $(\beta,c)$, such that :
		putting $T=T_i(c)+x$, the orbit $\gamma_i(\beta,c)$ has the asymptotic 
		$x\sim -k_i U^\frac{2-\alpha}{1-\alpha}.$  
	\end{lemma}
	
	\begin{figure}[htp]
		\begin{center}
			\includegraphics[scale=0.7]{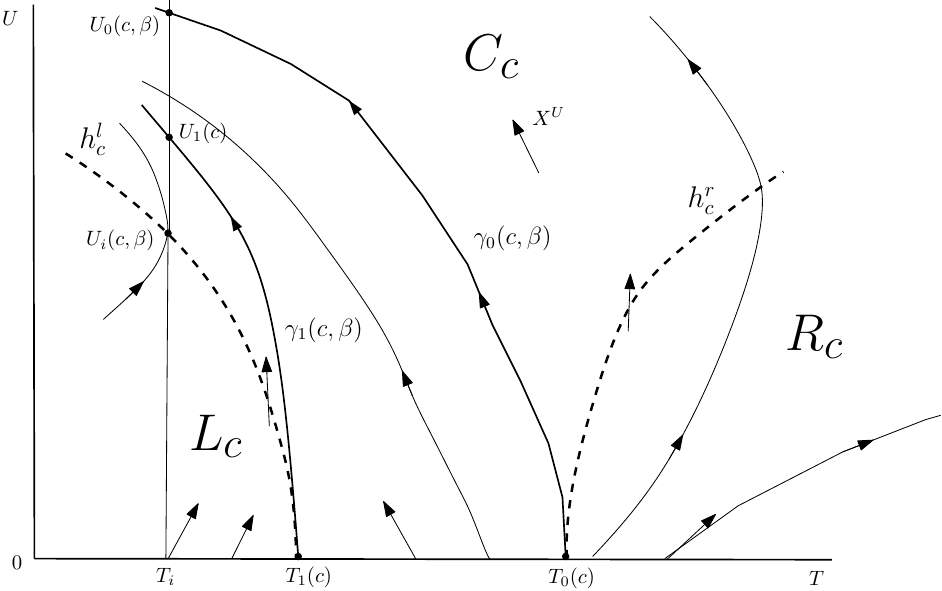}\\
			\caption{Phase portrait of $\tilde X^U$ for $c>\CJ$}\label{fig-field-X-U}
		\end{center}
	\end{figure}
	\subsection{Phase portrait of $\tilde X$}
	
	We now infer the phase portrait of $\tilde X$  from the one of $\tilde X^U,$ by taking its image by the map $\Phi^{-1}.$
	
	Outside $\{Z=0\},$ image of $\{U=0\},$
	the phase portraits are the same, up the diffeomorphism $\Phi^{-1}|_{Q^+}$ (which also preserves the time).  The only differences occur along the axis $\{Z=0\}.$ To obtain the asymptotics of the orbits along this axis  we will use the following elementary results:
	
	\begin{lemma}\label{lem-contact}
		(i) If a curve $\gamma$  has, at a point $(T_*,0),$ an asymptotic $T-T_*\sim kU^\delta$
		with $k\in \R -\{0\}$ and $\delta>0,$  then its image $\Phi^{-1}(\gamma)$ has an asymptotic $T-T_*\sim \frac{k}{1-\alpha}U^{\delta(1-\alpha)}.$\\
		(ii) As a consequence, for $\delta(1-\alpha)<1$ (for instance, for $\delta=1$), $\Phi^{-1}(\gamma)$ is tangent 
		to $\{Z=0\}$ with order $\frac{1}{\delta(1-\alpha)}$ on a side determined by the sign of $k$; for $\delta(1-\alpha)>1,$ $\Phi^{-1}(\gamma)$ is tangent to  the vertical with order $\delta(1-\alpha),$ on a side determined by the sign of $k.$
	\end{lemma}
	
	The axis  $\{Z=0\}$ is invariant: it contains  two singular points $(T_1(c),0),$ $ (T_0(c),0)$ and three regular orbits. 
	The most remarkable fact, following from Lemma \ref{lem-contact},  is  that  other orbits arrive  tangentially on the axis $\{Z=0\}$ in negative time, with a contact of order $\frac{1}{1-\alpha},$ save the two orbits ``arriving'' to the singular points $(T_1(c),0)$ and $(T_0(c),0)$. The contact is on the left for an orbit arriving at $(\bar T,0)$ for 
	$T_1(c)<\bar T <T_0(c)$ and on the right for $\bar T<T_1(c)$ or $\bar T> T_0(c).$ 
	At  $(T_1(c),0)$ and $(T_0(c),0)$ the orbits arrive transversality to $\{Z=0\},$ tangent to the vertical direction and on the left. We write  these orbits : $\gamma_1(\beta,c)$
	and $\gamma_0(\beta,c)$, respectively.
	This follows from Lemmas \ref{lem-transverse} and \ref{lem-contact} implying that the asymptotics is $T-T_j(c)\sim -\bar{k}_jZ^{(2-\alpha)}$),  for a $\bar{k}_j>0,$  for $j=0$ or $1.$ 
	These trajectories cut transversely the verticals $\{T=T_i\}$ at the points 
	$(T_i,Z_1(\beta,c))$ and 
	$(T_i,Z_0(\beta,c))$, respectively.
	
	A consequence is that \textit{the vector field $\tilde X$ does not satisfy the Cauchy property of uniqueness of trajectory:}  at any point $(T,0),$ with $T\not =T_1(c),T_0(c)$ arrive two trajectories : one is horizontal and the other is the image by $\Phi^{-1}$ of the trajectory of $X^U$. Inversely, it is easy to see that, starting at any point $(T_i,Z),$ with $Z>Z_1(\beta,c)$ ,  and going in negative time,  we arrive at a point $(\bar{T},0)$ in the finite (fast) time 
	$\tau_Z$ (equal to $\frac{Z^{1-\alpha}}{(1-\alpha)\beta}$ in the Heaviside case) and next, if $Z\not =Z_0(\beta,c),$ we slide  an infinite time along $\{\Z=0\},$ towards the point $(T_0(c),0).$ 
	
	The phase portrait is given in Figure \ref{fig-field-X-Z}.
	\begin{figure}[htp]
		\begin{center}
			\includegraphics[scale=0.7]{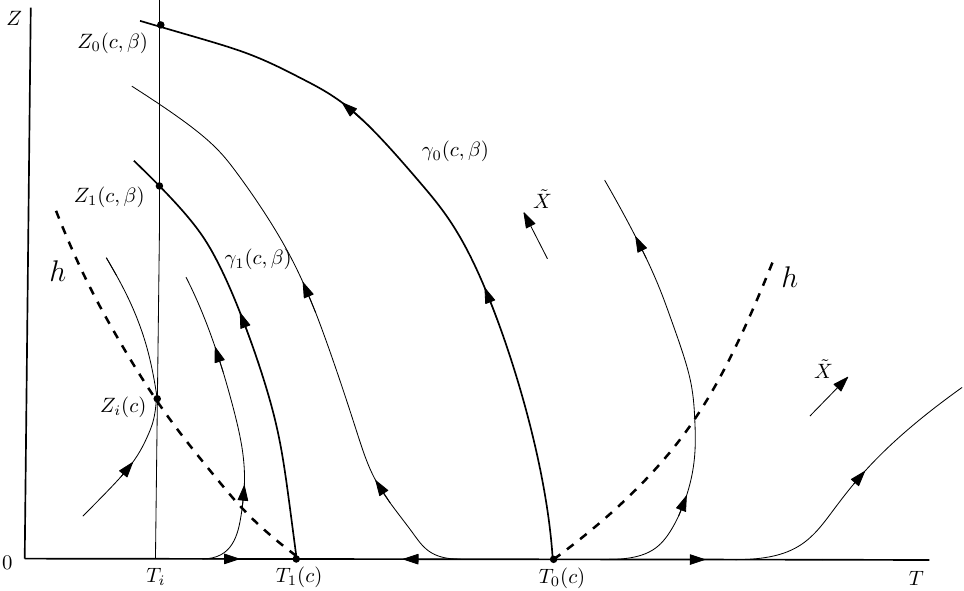}\\
			\caption{Phase portrait of $\tilde X$ for $c>\CJ$}\label{fig-field-X-Z}
		\end{center}
	\end{figure}
	%\eject
	The point $(T_0(c),0)$ is a degenerate repulsive node and the point $(T_1(c),0)$ is a degenerate saddle of $\tilde X.$ The basin (of repulsion) of $(T_0(c),0)$  is  at the right of the orbit $\gamma_1(\beta,c).$  For each $Z>Z_1(c)$ the trajectory through $(T_i,Z)$ tends toward $(T_0(c),0)$ for negative times. This time is finite if $Z=Z_0(\beta,c)$ and infinite otherwise. For $Z_1(\beta,c)<Z<Z_0(\beta,c)$ the trajectory arrives by the left side and is a monotonic graph $T(Z)$. For $Z>Z_0(\beta,c),$  the trajectory arrives by the rightarrow side after a unique bump (it has to cut the homoclinic branch $h^r_c$ only one time).
	The orbit $\gamma_0(\beta,c),$  that we say  of {\it finite time} type, separates the two types that we will call here {\it monotonic and bump orbit}, respectively, see Section \ref{main_section}.

	\begin{claim}\label{claim_solution} There is a solution if and only it holds:
		\begin{equation} \label{S}  T_1(c)\leqslant T_i\  \  {\rm or}\  \  Z_1(\beta,c)<1.
		\end{equation}
		This solution is of monotonic, finite time or bump type (see Section \ref{main_section}), respectively, depending on the position of $1:$ $Z_1(\beta,c) <1< Z_0(\beta,c),$ $Z_0(\beta,c)<1$ or $1=Z_0(\beta,c)$, respectively.
	\end{claim}
	
	We intend to address these issues in the next section by looking at the position  of $Z_1(\beta,c)$ and $Z_0(\beta,c)$ against the value $1,$ in function of the parameter $(\beta,c).$
	
	\subsection{The Chapman-Jouguet case}\label{case_CJ}
	In the previous subsections and in particular in Figures \ref{fig-field-X-U} and \ref{fig-field-X-Z}, we have implicitly assumed that $c>\CJ=\sqrt{2q_0}$. The limit case $c=\CJ$
	can be treated in the same way.   In this case, we have that $T_1(\CJ))=T_0(\CJ)=\sqrt{2q_0}.$
	
	As above, we begin with the phase portrait of $\tilde X^U.$ The two branches $h^l_{\CJ}$ and $h^r_{\CJ}$ start from the point $(0,\CJ)$  with a contact of order 
	$\frac{1}{1-\alpha}$. The two orbits $\gamma_i(\beta,\CJ)$ coincide now in a single orbit 
	$\gamma(\beta,\CJ)$ which is contained in the region $C_{\CJ}$ This follows from Lemma \ref{lem-contact}, but it is also clear for topological reasons. Then, we have that 
	$U_0(\beta,\CJ)=U_1(\beta,\CJ)$. From the phase portrait of $\tilde X^U$ we can deduce the one of $\tilde X$ as above. A possible solution, which exits if and only if $1>Z_0(\CJ),$
	is always of bump type. The phase portraits of $\tilde X^U$ and $\tilde X$ are presented in Figure~\ref{fig-field-c0}.
	\begin{figure}[htp]
		\begin{center}
			\includegraphics[scale=0.6]{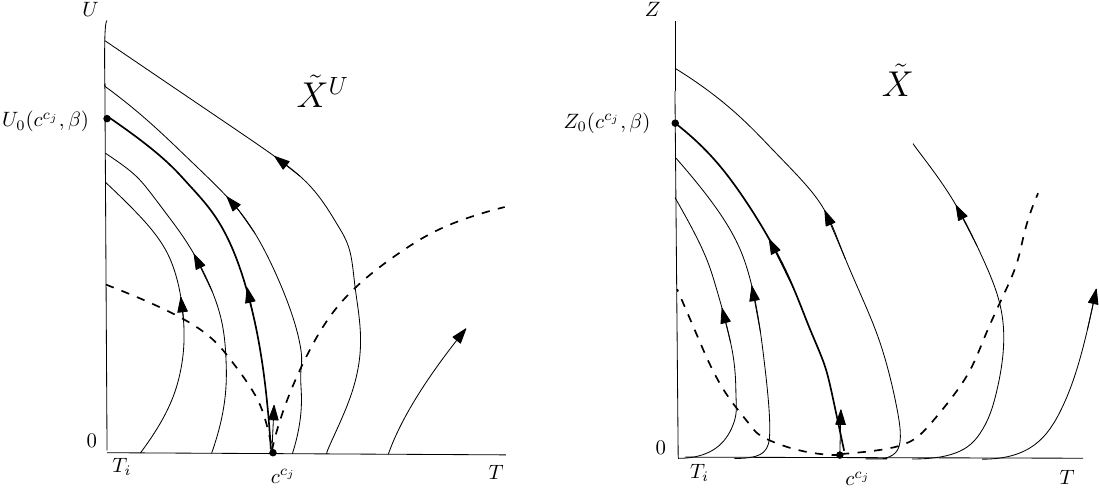}\\
			\caption{Phase portraits for $c=\CJ$}\label{fig-field-c0}
		\end{center}
	\end{figure}
	The vector field $\tilde X$ has a unique singular point $(\CJ,0).$   This point is a degenerate saddle-node singularity. The orbit $\gamma(\beta,\CJ)$ is an unstable separatrice which separates the saddle type sector from the node type sector. When $c$ increases, this point bifurcates into the pair of a degenerate saddle and a degenerate node
	appearing in Figure \ref{fig-field-X-Z}.

	\section{Existence  of a solution in function of $(\beta,c)$}\label{existence}
	We have in mind to study the nature of the trajectory in negative times of $\tilde X$,  passing through $(T_i,1 )$   in function of the parameter $(\beta,c)$. Our goal is to determine if this trajectory is a solution and identify its type. To emphasize that $\tilde X$  depends on the parameter $(\beta,c)$, we will write it as $\tilde X_{\beta,c}$ in this section and the next one.
	\subsection{Property for $c$ large}
	First, we fix any value of $\beta>0$  and look at how it depends on $c$. It is given by the elementary result:
	\begin{lemma}\label{lem-c-large}
		If $c\geqslant c_{\star}=\frac{T_i^2+2q_0}{T_i}$, we have that $T_1(c)\leqslant T_i $ and therefore the degenerate saddle point $(T_1(c),0)$ no longer belongs to $(T_i,+\infty)$. We are in the first case of the claim \ref{claim_solution}: $(T_i,1)$ belongs to  the basin of $(T_0(c),0)$ and the trajectory of this point, for negative times,  is a solution (of one type or the other).
	\end{lemma}
	\subsection{Properties for $\beta$ near $0$ and $+\infty$.}\label{limite_cases}
	We now fix any value of $c\geqslant \CJ$ and we look at the variation of $\beta$ from $0$ to $+\infty.$ 
	
	\subsubsection {\bf Slow analysis for $\beta\rightarrow 0_{+}$}\label{(a)}
	%\noindent {\bf  (a) Slow-fast analysis : $\beta\rightarrow 0.$ }
	%We begin to look at $\beta\rightarrow 0_+.$  
	We have a {\it slow-fast system} in dimension $2,$  with small parameter $\beta$, see \cite{DDMR} for detailed definitions and properties. As the parameter $c$ is fixed, we will usually skip it.
	
	For convenience, we consider the system  $-\tilde X_{\beta,c}$. The {\it limit equation 
		$-\tilde X_{c,0}$ } for $\beta =0$ has a very simple phase portrait: $h_c$ is a curve of zeros (called the {\it critical curve}). The dynamics outside this critical curve (namely the fast dynamics) is horizontal, without singularities. Here, we  just need to consider the vector field 
	$-\tilde X_{\beta,c}$   on  $Q^+=Q\cap \{Z >0\}.$ On $Q^+,$ the vector field  
	$-\tilde X_{\beta,c}$  is smooth. The two branches of $h_c$ are normally hyperbolic arcs of zeros, $h^r_c$ is attracting and $h^l_c$ repulsive. For any $\beta>0,$ small, we are interested in the orbit of
	$\Gamma_\beta$ of $-\tilde X_{\beta,c}$  starting at the point $(T_i,1).$ 
	\begin{figure}[htp]
		\begin{center}
			\includegraphics[scale=0.5]{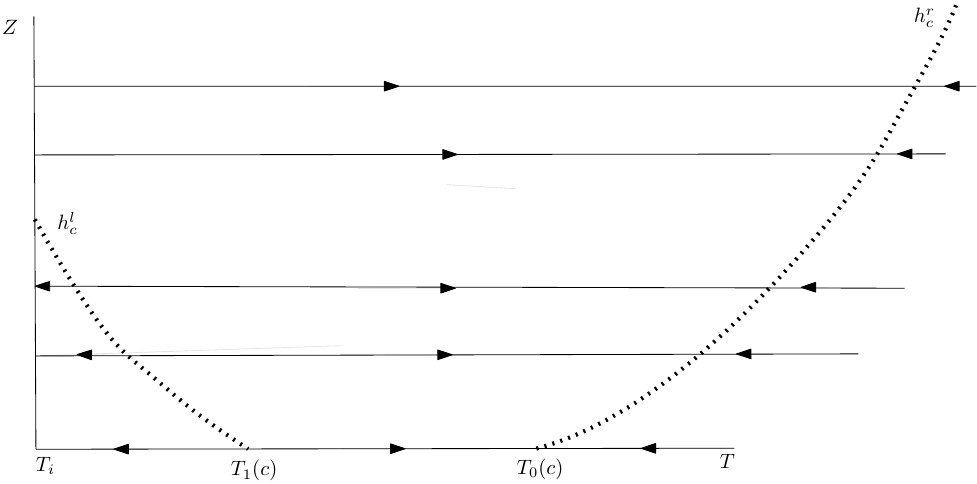}\\
			\caption{The limit vector field for $\beta=0$}\label{fig-limit} 
		\end{center}
	\end{figure}
	We want to prove that, for $\beta>0$ small enough, we will have $ 1>Z_0(\beta,c),$ 
	that is, we have a bump solution. 
	
	We will  give a very basic proof here, which will not use the technicalities of the theory of slow-fast systems as in \cite{DDMR}.
	
	Let $\{ T=h(Z)\}$ the equation of the branch $h^r_c.$ The function $h(Z)$ is smooth and $h'(Z)>0$  for $Z>0.$  We choose any interval $[Z_b,Z_a]$ such that $0<Z_b<1<Z_a.$ and associate with any  $\delta>0$  a tubular neighborhood 
	$$\mathcal{T}_\delta=\{Z_a\leqslant Z\leqslant Z_b, h(Z)-\delta\leqslant T \leqslant h(Z)+\delta\}.$$ 
	$\mathcal{T}_\delta$ is a thin curved rectangle with corners 
	$$a_\delta=(h(Z_a)-\delta,Z_a),\ \  a'_\delta=(h(Z_a)+\delta,Z_a)$$
	$$b_\delta=(h(Z_b)-\delta,Z_b),\ \  b'_\delta=(h(Z_b)+\delta,Z_a).$$ 
	We will note $[a_\delta,a'_\delta]$, $[b_\delta,b'_\delta]$, $[a_\delta,b_\delta]$  and 
	$[a'_\delta,b'_\delta]$ the four sides of $\mathcal{T}_\delta$
	(see Figure \ref{fig-tubular}).
	If $\delta>0$ is small enough, the following condition (R) is fulfilled:
	\begin{equation}
		\tag{R}  
		h(1)-\delta>h(Z_b)+\delta. 
	\end{equation}
	
	%Sorry, it is not the case in the Figure \ref{fig-tubular}!
	
	%
	%
	\begin{figure}[htp]
		\begin{center}
			\includegraphics[scale=0.8]{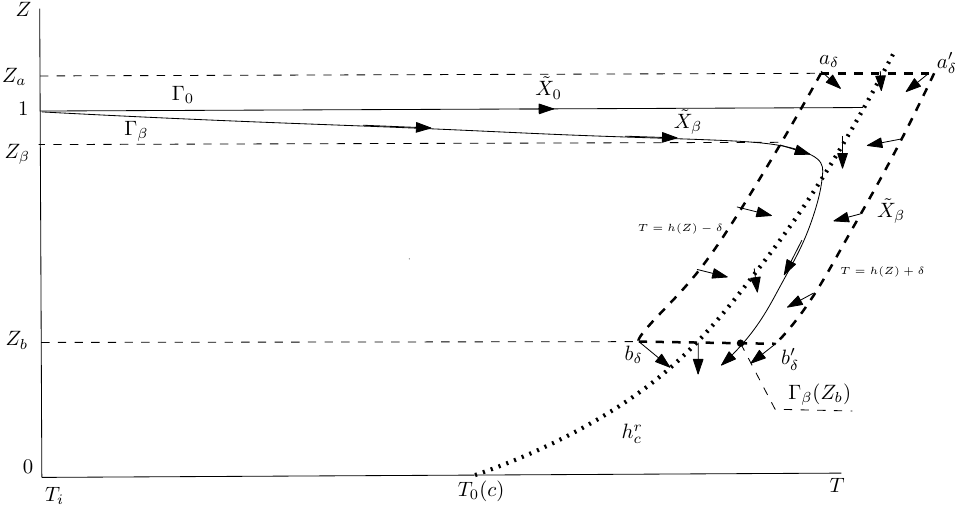}\\
			\caption{Trapping the solution for $\beta$ small.}\label{fig-tubular} 
		\end{center}
	\end{figure}
	Recall that any orbit of $-\tilde X_{\beta,c}$  through a point  $(T_i,\bar Z)$ and for any $\beta>0$  is a graph  above an interval of the vertical $Z$-axis. The following proposition proves that, for  $\beta>0$  small enough, the orbit $\Gamma_\beta$  through the point $(T_i,1)$ is a solution of bump type. As follows from the analysis presented in Section 1, it is sufficient to prove that $\Gamma_\beta$ cuts the homocline branch $h_c^r.$
	
	\begin{proposition}\label{prop-beta0}
		Choose any   $\delta>0,$ verifying the condition (R). We consider the orbit 
		$\Gamma_\beta$ of $-\tilde X_{\beta,c}$, starting at the point 
		$(T_i,1).$ If $\beta>0$ is small enough,  $\Gamma_\beta$ is defined at least above the interval 
		$[Z_b,1]$ as a graph of a smooth function $T=\varphi_\beta(Z).$ There is a $\bar Z_\beta\in(Z_b,1)$ such that $\varphi'_\beta(\bar Z_\beta)=0.$ As a consequence, $\Gamma_\beta$ cuts the homocline branch $h_c^r$ and then is a solution of bump type. This implies  that $Z_0(\beta,c)<1$ if $\beta>0$ is small enough.
	\end{proposition}
	\begin{proof} 
		Choose a $\delta>0$ verifying the condition (R) and consider the orbit $\Gamma_\beta$ through $(T_i,1)$.  It is the graph of a function $T=\varphi_\beta(Z),$ smooth for $Z>0.$ 
		If $\beta>0$ is small enough, $\Gamma_\beta $ contains an arc between $(T_i,1)$ and a point $(h(Z_\beta)-\delta,Z_\beta)$, located at the boundary of $\mathcal{T}_\delta$ below the point $a_\delta.$ Moreover $Z_\beta\rightarrow 1_-$ for $\beta\rightarrow 0$ and the arc of $\Gamma_\beta$ above $[Z_\beta,1]$ converges smoothly toward the horizontal. This means in particular that $\varphi_\beta$ is defined in a neighborhood of $1$, with $\varphi_\beta(Z_\beta)=h(Z_\beta)-\delta$,  and that $\varphi'_\beta(1)\rightarrow -\infty$ when $\beta\rightarrow 0.$
		
		We look now at the continuation of $\Gamma_\beta$ from  the point $(h(Z_\beta)-\delta,Z_\beta).$ First, we observe that  the rectangle $\mathcal{T}_\delta$ has a {\it trapping property:} for any $\beta>0,$ the vector field $-\tilde X_\beta$  is  transversely pointing inward along the three sides $[a_\delta,a'_\delta]$, $[a'_\delta,b'_\delta]$, 
		$[a_\delta,b_\delta]$  and transversely pointing outward along the last side 
		$[b_\delta,b'_\delta].$  Moreover, $-\tilde X_\beta$ has no singular point inside $\mathcal{T}_\delta.$ Then, it follows from the Poincar\'e-Bendixson theorem that  the trajectory entering the rectangle at the point 
		$(h(Z_\beta)-\delta,Z_\beta)\in [a_\delta,b_\delta]$ must reach a point in
		$[b_\delta,b'_\delta].$  This means that  $\varphi_\beta$ is defined on $[Z_b,1]$ and that $\varphi_\beta(Z_b)\in [h(Z_d)-\delta,h(Z_b)+\delta],$ and in particular it holds
		\begin{equation}\label{eq-Zb1}
			\varphi_\beta(Z_b)<h(Z_b)+\delta.
		\end{equation}
		Now, as $Z_\beta\rightarrow 1$ for $\beta\rightarrow 0,$ it follows from condition (R) that 
		\begin{equation}\label{eq-Zb2}
			h(Z_\beta)-\delta>h(Z_b)+\delta,
		\end{equation}
		if $\beta>0$ is small enough.
		
		By definition, $h(Z_\beta)-\delta =\varphi_\beta(Z_\beta).$  Then it follows from
		(\ref{eq-Zb1}) and (\ref{eq-Zb2}) that $\varphi_\beta(Z_b)<\varphi_\beta(Z_\beta).$
		As $Z_\beta>Z_b$,  the mean value theorem implies that there exists a 
		$\bar Z_1\in ]Z_b,Z_\beta[$ such that  $\varphi_\beta'(\bar Z_1)>0.$ As 
		$\varphi_\beta'(1)<0$ for $\beta>0$ small enough, it eventually follows from the intermediary value theorem that there exists a $\bar Z_\beta\in (\bar Z_1,1)\subset (Z_b,1)$ such that 
		$\varphi'_\beta(\bar Z_\beta)=0.$  This completes the proof.
	\end{proof}
	
	\begin{remark}
		Using the general theory of slow-fast  systems, it is possible to obtain additional properties of the flow when $\beta\rightarrow 0.$  For example, one can prove that $Z_0(\beta,c)$ and $Z_1(\beta,c)$ tends toward $Z_i(c)$ when $\beta\rightarrow 0.$
	\end{remark}
	\subsubsection{\bf Limit for $\beta\rightarrow +\infty$}\label{(b)}
	%\noindent{\bf (b) Limit for $\beta\rightarrow +\infty$}
	To look at the limit $\beta\rightarrow +\infty$,  it is better to return to the initial vector field which we now write $X_{\beta,c}$, and to the slow time $\xi.$ We also introduce the corresponding vector field 
	\begin{equation}\label{eq-XUbeta}
		X^U_{\beta,c}:= 
		\left\{
		\begin{array}{l}
			{\dot T} = \frac{1}{\beta}\Big[-cT+\frac{1}{2}T^2+q_0\Big(1-(1-\alpha)^{\frac{1}{1-\alpha}}U^{\frac{1}{1-\alpha}}\Big)\Big],  \\
			{\dot U} = \varphi(T).
		\end{array}
		\right.
	\end{equation}
	For $\beta=+\infty,$  we have a vertical vector field, without zeros: 
	$X^U_\infty= \varphi(T)\frac{\partial}{\partial U}.$  From this, we can deduce the following result:
	\begin{proposition}\label{prop-betainfty}
		Take any fixed $c\geqslant \CJ.$ If $\beta\rightarrow +\infty,$
		we have that :
		$$Z_1(\beta,c) \rightarrow +\infty \ \  {\rm and }\ \  Z_0(\beta,c) \rightarrow +\infty$$ 
	\end{proposition}
	\begin{proof}
		Fix a value $c\geqslant \CJ.$ Clearly, the vector field $X^U_{\beta,c}$ converges uniformly toward the vertical vector field 
		$\varphi(T)\frac{\partial}{\partial U}$ on each compact subset of $Q.$
		This obviously implies that $Z_1(\beta,c)$ and $Z_0(\beta,c)$ take arbitrarily large values when $\beta\rightarrow +\infty.$ But since these two functions are strictly increasing, as proved in Proposition \ref{prop-rotating} below,  they have to converge toward $+\infty.$
	\end{proof}

	\subsection{\ The transition map $Z(T,\beta,c)$ }
	\vskip5pt
	
	The vector field $\tilde X^U_{\beta,c}$ is transverse  to the horizontal axis $\{U=0\}.$ From  the claim \ref{claim(iv)} of Subsection \ref{tilde_U}, it follows that each trajectory  starting at a point of $[T_i,T_0(c)]\times \{0\}$ 
	cuts transversally the vertical axis $\{T=T_i\}$ after a finite time.  By continuity, this implies that there exists a $T(\beta,c)>T_0(c)$ such that  each trajectory  starting at a point of $[T_i,T(\beta,c))\times \{0\}$ 
	cuts also transversally the vertical axis $\{T=T_i\}$ after a finite time.
	
	This defines  a transition map from $[T_i,T(\beta,c))\times \{0\}$ to $\{T=T_i\}.$ Choosing  for convenience the variable $Z$ to parametrize the vertical axis, this transition map is a function $Z(T,\beta,c)$  of class $\big [\frac{1}{1-\alpha}\big]$ on the domain:
	$$\mathcal{Z}=\{(T,\beta,c) \ | \  c\in [\CJ,+\infty),\  \beta\in (0,+\infty), \ T\in [T_i,T(\beta,c) )\}.$$
	This follows from the Cauchy theorem and the implicit  function theorem.
	
	For each fixed $(\beta,c),$  the transition map $T\mapsto Z(T,\beta,c)$  is a diffeomorphism of class $\big [\frac{1}{1-\alpha}\big],$ preserving the orientation. This means that, for all $(T,\beta,c)\in \mathcal{Z}$,
	\begin{equation}\label{eq-PartialZ-PartialT}
		\frac{\partial Z}{\partial T}(T,\beta,c)>0.
	\end{equation}
	The more interesting values of  $Z(T,\beta,c)$  are $Z_1(\beta,c)=Z(T_1(c),\beta,c)$ and $Z_0(\beta,c)=Z(T_0(c),\beta,c).$  Since $Z(T,\beta,c)$  is of class $\big [\frac{1}{1-\alpha}\big]$, we also have that $Z_1(\beta,c)$ and $Z_0(\beta,c)$ are of class $\big [\frac{1}{1-\alpha}\big].$

	\subsection{Rotating property}\label{rotating_property}
	
	We want  to prove a   {\it rotating property} for the vector field  $\tilde X^U_{\beta,c}$ (see \cite{DU}), in function of the parameters $\beta$ and $c$.
	The general definition of the rotating property is as follows:
	\begin{definition}\label{def-rotating}
		Let $X_\varepsilon$ a $1$-family of vector fields  with  parameter 
		$\varepsilon\in (-\varepsilon_0,\varepsilon_0),$ for some small $\varepsilon_0>0,$ defined on an open connected subset $W\subset \R^2.$ We assume that this family is at least 
		$C^1$, i.e., that the components are $C^1$ as functions of $(m,\varepsilon)\in W\times (-\varepsilon_0,\varepsilon_0)$. 
		We say that $X_\varepsilon$ has {\rm rotating property in function of $\varepsilon$} if for all $m\in W$ and all $\varepsilon,$ the vectors $X_\varepsilon(m)$ and $\frac{\partial X_\varepsilon}{\partial \varepsilon}(m)$ are linearly independent. Then, 
		${\rm Det}(X_\varepsilon(m),\frac{\partial X_\varepsilon}{\partial \varepsilon}(m))$
		has a constant sign on $W.$
		% the family is direct rotating if this sign is positive, and clockwise rotating if it is negative. Equivalently, if $\varepsilon<\varepsilon',$ we rotate in the direct direction  from $X_\varepsilon(m)$ to $X_{\varepsilon'}(m)$
		% in the first case and in the clockwise direction in the second case. 
	\end{definition}
	An elementary basic result is the following:
	\begin{lemma}\label{lem-rotating}
		Let $W$ be an open connected subset of $\R^2$ and  $X_\varepsilon$  an $\varepsilon$-family  of vector field s, of class 
		$C^k,$ for some $k\geqslant 1,$ defined on a neighborhood $Q$ of $\overline W$ (the closure of $W$). We assume that $X_\varepsilon$ is on $W$  a rotating family as in Definition \ref{def-rotating}.   Let be $p\in \overline W$  and $\gamma_\varepsilon(p)$ the positive orbit of $X_\varepsilon$ in  $Q$, through $p.$  We assume that $\gamma_0$ is contained in $\overline W,$ but is not entirely included in $\partial W=\overline W\setminus W.$ 
		
		Let be $q\in \gamma_0(p)$ with $q\not =p.$ We choose a transverse section $\Sigma$ to $X_0$  through $q,$  with a smooth coordinate  $s,$   such that $q=\{s=0\}.$
		
		Assume that $\varepsilon_0$  is small enough, such that $\Sigma$ is also transverse to  $X_\varepsilon$ for all $\varepsilon$ and such that $\gamma_\varepsilon(p)$ cuts $\Sigma$ at a point of coordinate $s=P(\varepsilon),$  for all $\varepsilon$ (by assumption $P(0)=0$).
		Then, the function  $s=P(\varepsilon)$ is $C^k$ and $P'(0)\not =0:$ 
		we have that $P'(0)>0$ if the vector $\frac{\partial X_\varepsilon}{\partial \varepsilon}(q)$ is oriented in the $s$-direction and $P'(0)<0$ if not.
		See Figure \ref{fig-rot-lem}.
	\end{lemma}
	
	\begin{figure}[htp]
		\begin{center}
			\includegraphics[scale=0.6]{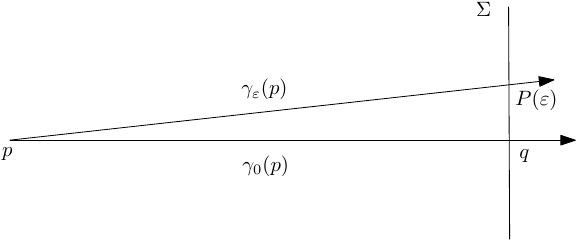}\\
			\caption{The function $P(\varepsilon)$}\label{fig-rot-lem} 
		\end{center}
	\end{figure}
	
	\begin{proof}
		Let us remark that the result does not change if we replace $X_\varepsilon$ by a family 
		$C^k$- conjugated (i.e. changed by a $C^k$- family of diffeomorphisms) or $C^k$-equivalent  (i.e. multiplied by a $C^k$- family of positive functions). We will use this possibility several times. We will keep the name $X_\varepsilon$ for the resulting vector field family.
		
		First, we use the flot box theorem (see, e.g., \cite{DLA}) in a neighborhood 
		$\mathcal{T}$ of $\gamma_0(p)$ in $Q$, with local 
		$C^k$ coordinates $(x,y)$ in which, up a  $C^k$-equivalence, we have that $X_0=\frac{\partial}{\partial x},$ $p=(0,0)$ and $\Sigma=\{1\}\times [-1,1].$ In $\mathcal{T}$,  the family takes the form (for $\varepsilon_0>0$ small enough and up to $C^k$ families of diffeomorphisms and equivalences):
		$$
		X_\varepsilon := \Big(1+\varepsilon A(x,y,\varepsilon)\Big)\frac{\partial}{\partial x}+\varepsilon B(x,y,\varepsilon)\frac{\partial}{\partial y}.
		$$
		As $\gamma_0\subset \bar W$ and changing $y$ if necessary by $-y$, we have that 
		$B(x,0,0) \geqslant 0$ (this means that  $\frac{\partial X_\varepsilon}{\partial \varepsilon}$ is oriented in the $y$-direction  along the $x$-axis).
		Moreover, the hypothesis that $\gamma_0$ is not completely contained in 
		$\partial W$ implies that $B(x,0,0)\not \equiv 0$ and then that:
		\begin{equation}\label{eq-intB}
			\int_0^1B(x,0,0)dx>0.
		\end{equation}
		If $\varepsilon_0>0$ is small enough  we have that
		$1+\varepsilon A(x,y,\varepsilon)>0,$  and then, up  to $C^k$-equivalence,
		we have that:
		\begin{equation}
			X_\varepsilon := \frac{\partial}{\partial x}+\varepsilon\bar B(x,y,\varepsilon) \frac{\partial}{\partial y},
		\end{equation}
		for a new $C^k$ function $\bar B(x,y,\varepsilon)=B(x,y,0)+O(\varepsilon).$   The equation of trajectories takes the form:
		
		\begin{equation}\label{eq-xy}
			\tilde X := 
			\left\{
			\begin{array}{l}
				{\dot x} = 1  \\
				{\dot y} =\varepsilon \bar B(x,y,\varepsilon).
			\end{array}
			\right.
		\end{equation}
		
		We look at the trajectory of $p$. For it, we have that $x\equiv t$ and then we can reduce the equation to a single line for a function $y_\varepsilon(x)$ whose  $\gamma_\varepsilon(p)$ is the graph:
		
		\begin{equation}\label{eq-y}
			\frac{dy_\varepsilon}{dx}(x)=\varepsilon \bar B(x,y_\varepsilon,\varepsilon).
		\end{equation}
		As $y_\varepsilon(0)=0$ it follows from (\ref{eq-y}) that  $y_\varepsilon(x)=O(\varepsilon).$ Reporting this in (\ref{eq-y}),  we have that 
		\begin{equation}\label{eq-y-must}
			\frac{dy_\varepsilon}{dx}(x)=\varepsilon B(x,0,0)+O(\varepsilon^2),
		\end{equation}
		where  the remainder $O(\varepsilon^2)$ is now a function of $(x,\varepsilon)$, of class $C^k.$ This equation (\ref{eq-y-must}) implies that 
		\begin{equation}\label{eq-y-must-2}
			y_\varepsilon(x)=\varepsilon\int_0^xB(\tau,0,0)d\tau +O(\varepsilon^2).
		\end{equation}
		As $P(\varepsilon)=y_\varepsilon(1)$, we obtain that  $P(\varepsilon)$ is 
		$C^k$ and, using (\ref{eq-intB}), that: 
		$$P'(0)=\int_0^1 B(\tau,0,0)d\tau >0. $$
		Clearly, the sign of $P'(0)$ changes if we change the orientation of $\Sigma$ (i.e. of the $y$-axis).
		This finishes the proof.
	\end{proof}
	
	\begin{remark}
		Lemma \ref{lem-rotating} can be proved using the general theory of 
		first order variational equations (see \cite{CL}).  This theory would give an intrinsic expression for the derivative $P'(0).$
	\end{remark}
	
	We now return to the family  $\tilde X^U_{\beta,c}.$ 
	The vector field  $\frac{\partial \tilde  X^U_{\beta,c}}{\partial \beta}(m)=\varphi(T)\frac{\partial}{\partial U}$ is transverse to  $\tilde X^U_{\beta,c}$   on $C_c.$
	Moreover,  the pair $\big(\tilde  X^U_{\beta,c}(m) ,\frac{\partial \tilde  X^U_{\beta,c}}{\partial \beta}(m)\big)$  has a clockwise orientation.  
	Then, the family $\tilde X^U_{\beta,c}$  has the rotating property in the clockwise direction on the open connected set $C_c$, in function of the parameter $\varepsilon=\beta-\beta_0$ for any $\beta_0\in(0,+\infty).$ This clockwise orientation implies that $\frac{\partial \tilde  X^U_{\beta,c}}{\partial \beta}(m)$ is oriented as the $Z$-axis in comparison to the orbit of $\tilde  X^U_{\beta,c},$ at any  point $m$ of $C_c.$
	
	We consider the trajectory $\gamma_{\beta_0,c}( T)$  of   $\tilde X^U_{\beta_0,c}$  starting at any point $(T,0)$,  for $ T\in [T_1(c),T_0(c)].$  This trajectory is contained in the closure $\bar C_c$ and cuts the boundary $\partial C_c$ only at the starting point  $(T,0).$ Thus, a direct consequence of
	Lemma \ref{lem-rotating} is the following:
	\begin{corollary}\label{cor-rotating} 
		For $T\in [T_1(c),T_0(c)]$ and any $(\beta,c)$, it holds
		$\frac{\partial Z}{\partial \beta}(T,\beta,c)>0$.
	\end{corollary}
	
	Applying this to $T=T_1(c)$ and $T=T_0(c),$ we have the following result:
	\vskip5pt
	\begin{proposition}\label{prop-rotating}
		For $i=0,1$ and any fixed value of $c\in [\CJ,+\infty),$ the two functions $\beta\mapsto Z_i(\beta,c)$  are of class $\big[\frac{1}{1-\alpha}\big]$ in $(\beta,c).$ Moreover we have that 
		$$\frac{\partial Z_i}{\partial \beta}(\beta,c)>0$$
	\end{proposition}
	\begin{figure}[htp]
		\begin{center}
			\includegraphics[scale=0.7]{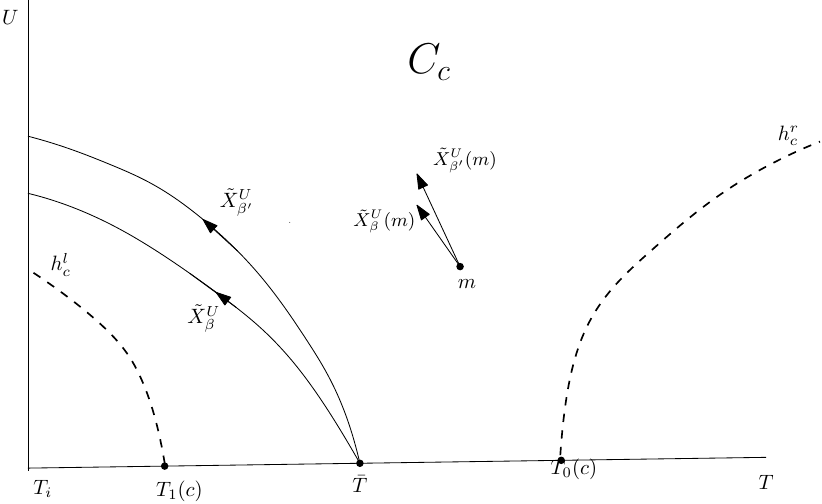}\\
			\caption{Rotating property in function of $\beta$ with $\beta'>\beta$.}\label{fig-rotating}
		\end{center}
	\end{figure}
	The vector field  $\frac{\partial \tilde  X^U_{\beta,c}}{\partial c}(m)=-T\frac{\partial}{\partial T}$ is transverse to  $\tilde X^U_{\beta,c}$   on the whole domain $Q.$
	Moreover,  the pair $\big(\tilde  X^U_{\beta,c}(m),\frac{\partial \tilde  X^U_{\beta,c}}{\partial c}(m)\big)$  has a direct orientation.  
	Then, the family   $\tilde X^U_{\beta,c}$  has the rotating property in the direct direction on the  open connected set $Q$, in function of the parameter $\varepsilon=c-c_0$ for any $c_0\in [\CJ,+\infty)$. This direct orientation implies that $\frac{\partial \tilde  X^U_{\beta,c}}{\partial c}(m)$ is oriented as the opposite direction of the  $Z$-axis in comparison to the orbit of $\tilde  X^U_{\beta,c},$ at any  point $m$ of $Q.$

	We consider the trajectory $\gamma_{\beta,c_0}( T)$  of   $ \tilde X^U_{\beta,c_0}$  starting at any point $(T,0)$,  for $ T\geqslant T_i.$  This trajectory is contained in $Q$ and cuts the boundary $\partial Q$ only at the starting point  $(T,0).$ Then, a direct consequence of
	Lemma \ref{lem-rotating} is the following:
	\begin{corollary}\label{cor-rotating} 
		For $T\geqslant T_i$ and any $(\beta,c)$ we have that: 
		$\frac{\partial Z}{\partial c}(T,\beta,c)<0.$
	\end{corollary}

	\section{Bifurcation diagram in the $(\beta,c)$-space}\label{bifurcation_diagrams}
	
	We can now deduce from the previous results a diagram of bifurcation for the solution.
	\subsection{The bifurcation curves $\beta=\beta_1(c)$ and $\beta=\beta_0(c)$.}
	\begin{theorem}\label{th-beta0(c)}
		There exists a  function  $\beta_0(c) :[\CJ,+\infty)\rightarrow \R^+,$ 
		of class $\big[\frac{1}{1-\alpha}\big],$ such that:
		
		\begin{enumerate}[label=(\roman*), wide, labelwidth=!, labelindent=0pt]
			
			\item there is a special solution (i.e., a finite-time solution) for $\beta=\beta_0(c)$;
			\item there is a non-monotonic solution with a bump for $0<\beta<\beta_0(c)$. 
		\end{enumerate}
		
	\end{theorem}
	\begin{proof}
		Take any $c \geqslant\CJ.$ From Proposition \ref{prop-beta0}, we know that
		$Z_0(\beta,c)<1$ is $\beta>0$ is small enough. From Proposition \ref{prop-rotating} we know, that $Z_0(\beta,c)$ is of class  $\big[\frac{1}{1-\alpha}\big]$  with a positive partial derivative in function of $\beta$,  and from Proposition 
		\ref{prop-betainfty} we know that $Z_0(\beta,c)\rightarrow +\infty$ if $\beta\rightarrow +\infty.$  Then, by the implicit function theorem, there exists a unique function  
		$\beta_0(c)$ of class  $\Big[\frac{1}{1-\alpha}\Big]$ such that $Z_0(\beta_0(c),c)=1$. 
		Clearly, if $0<\beta<\beta_0(c),$ we have that $Z_0(\beta,c)<1$ and  there is a solution of bump type.
	\end{proof}
	
	\begin{theorem}\label{th-beta1(c)}
		(i) Let $c\geqslant c_{\star}=\frac{T_i^2+2q_0}{T_i}$, there is a solution  for any $\beta>0$ and, if $\beta > \beta_0(c)$, this solution  is of monotonic type;\\
		(ii) there exists a  function 
		$\beta_1(c) :\big[\CJ,c_{\star}\big)\rightarrow \R^+,$ of class $\big[\frac{1}{1-\alpha}\big],$ with the following properties: 
		\begin{enumerate}[label=(\alph*)]
			\item For  all $c\in (c^{c_j},c_*)$ we have that 
			$\beta_0(c)<\beta_1(c)$  and $\beta_0(c^{c_j})=\beta_1(c^{c_j}).$ We call
			$\beta^{c^{c_j}}$ this common value;
			
			\item for $\beta_0(c)<\beta<\beta_1(c)$, there is a solution of  monotonic type;  
			
			\item there is no solution if $c\in\big[\CJ,c_{\star}\big)$ and 
			$\beta\geqslant \beta_1(c).$ 
			
			\item  $\beta_1'(c)>0$ for $c\in \big(\CJ,c_{\star}\big).$ Moreover $\beta_1(c)\rightarrow +\infty$ when $c \rightarrow c_{\star}$. 
		\end{enumerate}
	\end{theorem}
	\begin{proof}
		The point $(i)$ is a direct consequence of Lemma \ref{lem-c-large} and the fact that  if 
		$\beta > \beta_0(c)$, then $1<Z_0(\beta,c).$  
		
		To prove the existence of the function 
		$\beta_1(c)$ with properties $(a),(b)$ and $(c)$, we use the same arguments as in Theorem \ref{th-beta0(c)}. That $\beta_0(c)<\beta_1(c)$ for any 
		$c\in (c^{c_j},c_*)$ comes from the fact that,  if $\beta$ increases,  then the point $Z_0(c)$ crosses the value $1$ before the point $Z_1(c).$ 
		For $c=c^{c_j},$ we have that $Z_0(c^{c_j})=Z_1(c^{c_j})$ and then that
		$\beta_0(c^{c_j})=\beta_1(c^{c_j}).$ 
		
		In order to prove the first claim of $(d)$,  we use that the function 
		$\beta_1(c)$ is an implicit solution of the equation:
		$Z_1(\beta_1(c),c)=Z(T_1(c),\beta_1(c),c)= 1.$ Differentiating it gives that 
		$$\frac{\partial Z}{\partial T}T'_1(c)+\frac{\partial Z}{\partial c}+\frac{\partial Z}{\partial \beta}\beta'_1(c)=0,$$
		where the partial derivatives are evaluated at $(T_1(c),\beta_1(c),c)$.
		We have that $\frac{\partial Z}{\partial T}>0,$ $\frac{\partial Z}{\partial c}<0,$
		$\frac{\partial Z}{\partial \beta}>0$ and $T'_1(c)<0.$ Then 
		$\beta_1'(c)>0$ for $c\in (\CJ,c_*).$  
		
		We now prove the second claim of $(d)$. 
		Take any value $\beta>\beta^{\mathit{cj}}.$ As $\beta^{\mathit{cj}}=\beta_1(\CJ),$ we know that
		$Z_1(\beta,\CJ)>1.$ On the other side, we have that $Z_1(\beta,c_*)=0.$ 
		As above, we have that 
		$$\frac{\partial Z_1}{\partial c}=\frac{\partial Z}{\partial T}T'_1(c)+\frac{\partial Z}{\partial c}<0.$$
		Then, there exists a unique $c_1(\beta)$ such that $\beta_1(c_1(\beta))=\beta$ (this means that the function $\beta\mapsto c_1(\beta)$ is the inverse of the function $\beta_1(c)$). Now, if 
		$c\in [c_1(\beta),c_*)$ we have that $\beta_1(c)\geqslant \beta$.  As this is true for any $\beta > \beta^{\mathit{cj}}$, this means that $\beta_1(c)\rightarrow +\infty$ when $c\rightarrow c_*.$ 
	\end{proof}
	\subsection{Behaviour near $c=\CJ$.}
	Let us recall that $\beta_0(c^{c_j})=\beta_1(c^{c_j})$ and that we write $\beta^{\mathit{cj}}$ this common value.
	\begin{proposition}\label{prop-global}
		The  union of the graphs of $\beta_0$ and $\beta_1$ with the point
		$(\beta^{\mathit{cj}},\CJ) $ is a simple curve $(B)$ in the parameter space, of class $\big[\frac{1}{1-\alpha}\big],$  diffeomorphic to $\R$  and tangent to the axis $\{c=\CJ\}.$ This contact is quadratic if $\alpha\geqslant \frac{1}{2}.$ 
	\end{proposition}
	\begin{proof}
		As the graphs of $\beta_0$ and $\beta_1$ are simple curves above for $c>\CJ,$ it is sufficient to prove that $(B)$ is also a simple curve in a small neighborhood of the point 
		$(\beta^{\mathit{cj}},\CJ)$.
		
		The equation of $(B)$ is given by the elimination of $T$ between the two equations: 
		\begin{equation}\label{elimination} 
			\left\{
			\begin{array}{l}
				P_c(T)=-cT+\frac{1}{2}T^2+q_0 =0, \\[0.5mm]
				Z(T,\beta,c)-1=0.
			\end{array}
			\right.
		\end{equation}
		It follows from Corollary \ref{cor-rotating},   applied to the limit case $c=\CJ,$ 
		that $$\frac{\partial Z}{\partial \beta}(\CJ,\beta,\sqrt{2q_0})>0.$$
		Then, for 
		$(T,\beta,c)\sim (\sqrt{2q_0},\beta^{\mathit{cj}},\CJ)$, we can solve the second equation for an implicit function 
		$\beta=\beta(T,c)$ with $\beta(\CJ,\CJ)=\beta^{\mathit{cj}},$ 
		defined in a neighborhood $W$ of $(T,c)=(\sqrt{2q_0},\CJ).$
		As $\frac{\partial Z}{\partial T}+\frac{\partial Z}{\partial \beta}\frac{\partial \beta}{\partial T}=0$, we can choose $W$ small enough  such that $\frac{\partial Z}{\partial \beta}>0$  and then such that 
		$\frac{\partial \beta}{\partial T}(T,c)< 0$ on $W$. 
		
		On the other side, it is easy to verify that the first equation defines  for $c\geqslant \CJ$ an arc of  hyperbola, which is a graph of the function  
		$$c(T)=\frac{1}{2}T+\frac{q_0}{T}$$   
		restricted to $T\in [T_i,+\infty).$ This function, $c(T)$, has a minimum at $T=\sqrt{2q_0}.$ Finally, $(B)$ near the point $(\beta^{\mathit{cj}},\CJ)$ is a $T$-parametrized curve:
		\begin{equation}\label{B_parametrized} 
			(B):=
			\left\{
			\begin{array}{l}
				c=c(T), \\ [0.5mm]
				\tilde\beta(T)=\beta(T,c(T)).
			\end{array}
			\right.
		\end{equation}
		This curve passes through  the point $(\sqrt{2q_0},\bar \beta)$ for $T=\sqrt{2q_0}.$
		We have that $\tilde\beta'(T)=\frac{\partial \beta}{\partial T}(T,c(T))+\frac{\partial \beta}{\partial c}(T,c(T))c'(T)$.
		In particular, for $T=\sqrt{2q_0},$ we have that $c'(\sqrt{2q_0})=0$ and 
		$\frac{\partial \beta}{\partial T}(\sqrt{2q_0},\sqrt{2q_0})<0$. 
		Then, $\tilde\beta'(\sqrt{2q_0})<0,$ which proves that $(B)$ is a simple curve in a neighborhood of $(\sqrt{2q_0},\bar \beta).$ 
		
		If $\alpha\geqslant 2$, all the functions are at least of class $C^2.$  Then, we can invert the function $\tilde \beta (T)$ as a function $T(\beta)$ of class   $\mathcal{C}^2,$ such that $T(\beta^{\mathit{cj}})=\sqrt{2q_0}$ and $T'(\beta^{\mathit{cj}})<0.$
		Locally, the curve $(B)$ is the graph of the function $\tilde c(\beta)=\beta\ c(T(\beta))$. We have that 
		$$\tilde c'(\beta)= c'(T(\beta))T'(\beta)\ \ {\rm  and}\  \    
		\tilde c''(\beta)=c''(T(\beta))(T'(\beta))^2+c'(T(\beta))T''(\beta).$$
		For $\beta=\beta^{\mathit{cj}}$,  this gives that $ c'(T(\beta^{\mathit{cj}})=0.$ Then, 
		$$\tilde c'(\beta^{\mathit{cj}})=0 \  \  {\rm and}\  \
		\tilde c''(\beta^{\mathit{cj}})=c''(T(\beta^{\mathit{cj}})(T'(\beta^{\mathit{cj}}))^2>0,$$
		which means that the function $\tilde c$ has a quadratic minimum at $\beta^{\mathit{cj}}$. 
	\end{proof}

	\subsection{Behavior of $\beta_0(c)$ for $c\rightarrow +\infty$.}\label{large}
	In general, little is known  about the behavior of the function 
	$\beta_0(c)$ in the large.
	The simple proof given for $\beta_1(c)$ does not work in an obvious way, since we do not know the sign of the term $\frac{\partial Z}{\partial T}T'_0(c)+\frac{\partial Z}{\partial c}$ for all $c.$ Here, we introduce an additional technical assumption which is obviously verified by the Heaviside function and the Arrhenius law:\\
	\noindent{\bf Hypothesis $(H_\infty)$:} {\it assume that the function 
		$\psi(V)=\varphi(\frac{1}{V})$ is smooth at $V=0$ with a value $\psi(0)=A>0.$} 
	
	We  know that
	$\beta'_0(c)\rightarrow -\infty$ for $c\rightarrow \CJ$ as consequence of Proposition \ref{prop-global}.  The following result is about the asymptotic behavior of $\beta_0(c)$ at infinity:
	\begin{proposition}\label{prop-c-infty}
		Assume that  $(H_\infty)$ is verified. If $c\rightarrow +\infty,$ then $\beta_0(c)\rightarrow +\infty$.
	\end{proposition}
	\begin{proof}
		The proof  is rather long and uses
		again the theory of slow-fast systems, but now in a slightly deeper way than in Subsection 2.2 for $\beta\rightarrow 0.$ 
		In order to use \cite{DDMR}, we need to have  $C^\infty$ (smooth) differential systems. Then, we choose any $\delta$ such that $0<\delta<1,$  and we restrict $\tilde X_{\beta,c}$ to $Q_\delta=\{ T\geqslant T_i\  \  Z\geqslant \delta\}$. {\it Clearly 
			$\tilde X_{\beta,c}$ is smooth on $Q_\delta.$ }
		We will proceed in  several steps.
		%{\it We will assume everywhere that $\beta>0.$}
		\vskip5pt\noindent{\bf Step 1: Reduction to the search of solutions with a bump}.
		We recall that any orbit (image of a trajectory, forgetting the integration time) of 
		$-\tilde X_{\beta,c}$ in $Q_\delta,$ starting at a point  of $\{T=T_i\}$ is a graph in $Z$, in particular the orbit $\Gamma(\beta,c)$ starting at the point $(T_i,1)$ is a graph $T=\gamma(Z,\beta,c)$ over 
		$[\delta,1]$
		with $\gamma(1,\beta,c)=T_i.$

		The main result of this subsection is the following lemma:
		\begin{lemma}\label{lem-rel-position}
			For any $\beta>0$, there exists a $c(\beta)$  such that, if $c> c(\beta)$, then 
			the orbit $\Gamma(\beta,c)$ of $-\tilde X_{\beta,c}$
			starting at  $(T_i,1)$ is of bump type.
		\end{lemma}
		The next steps are devoted to the proof of Lemma \ref{lem-rel-position}. To complete Step 1, let us show that Lemma \ref{lem-rel-position} implies Proposition \ref{prop-c-infty}: Take any $\beta>0,$ by definition of $\beta_0(c),$  for each $c>c(\beta)$ we have that $\beta_0(c)\geqslant \beta$.  In other words, \textit{for any $\beta>0$ there exists a $c(\beta)$ such that $c>c(\beta)$ implies that $\beta_0(c)>\beta.$}
		This means that $\beta_0(c)\rightarrow +\infty$ if $c\rightarrow +\infty.$ 
		
		\vskip5pt\noindent{\bf Step 2: $\tilde X_{\beta,c}$ as a slow-fast system in the small parameter 
			$\varepsilon=\frac{1}{c}.$}
		As we want to consider $c\rightarrow +\infty$, it is convenient to introduce the small positive parameter $\varepsilon=\frac{1}{c}\rightarrow 0_+$ and the slow fast system:
		\begin{equation}\label{eq-XUvarepsilon}
			X_{\beta,\varepsilon}=-\frac{1}{c}\tilde X_{\beta,c} := 
			\left\{
			\begin{array}{l}
				{T'} = T-\varepsilon\big(\frac{1}{2}T^2+q_0(1-Z)\big), \\[1mm]
				{Z'} =-\varepsilon\beta \varphi(T)Z^\alpha.
			\end{array}
			\right.
		\end{equation}
		The division of $\tilde X_{\beta,c}$ by $c$ is equivalent to changing the time $\tau$ by the time 
		$s=c\tau.$ 
		%For a function $f(s)$ we write: $\frac{df}{ds}=f'.$
		The limit equation is the trivial vector field $T\frac{\partial}{\partial T}$ which has no zero. But it will be interesting to consider the behavior of trajectories of $X_{\beta,\varepsilon}$
		for $T\rightarrow +\infty.$ 
		\vskip5pt\noindent{\bf Step 3: Looking near $T=+\infty.$}
		To overcome the problems associated with non-compactness, we compactify $X_{\beta,\varepsilon}$ in the $T$-direction by introducting the coordinate $V=\frac{1}{T}.$
		In the chart $(U,V)$, the  system $X_{\beta,\varepsilon}$ reads:
		\begin{equation}\label{eq-XZV}
			X_{\beta,\varepsilon} := 
			\left\{
			\begin{array}{l}
				{V'} = -V+\varepsilon\big(\frac{1}{2}+q_0(1-Z)V^2\big), \\[2mm]
				{Z'} =-\varepsilon\beta \varphi(\frac{1}{V})Z^\alpha.
			\end{array}
			\right.
		\end{equation}
		The domain $Q_\delta$ given by $Q_\delta=\{0<V<V_i=\frac{1}{T_i}\}, Z\geqslant \delta\}$ is compactified for $T\rightarrow +\infty$ by the smooth extension to $V=0$. The orbit 
		$\Gamma(\beta,c)$ is now called $\Gamma(\beta,\varepsilon).$  It is a graph 
		$V=\tilde\gamma(Z,\beta,\varepsilon)$ over the $Z$-interval $[\delta,1].$ As we go from 
		$\gamma(Z,\beta,c)$ to $\tilde\gamma(Z,\beta,\varepsilon)$ by a difffeomorphism of the target space, the critical points are preserved and \textit{$\Gamma(\beta,c)$ will be a bump type solution if  there exists a $\bar Z\in (\delta,1)$ such that 
			$\frac{\partial \tilde\gamma}{\partial Z}(\bar Z,\beta,\varepsilon)=0.$}
		The hypothesis $(H_\infty)$ allows to apply the theory of slow fast systems in a neighborhood of the axis $\{ V=0\}.$ 
		
		\vskip5pt\noindent{\bf Step 4: Center manifold}.
		We consider $X_{\beta,\varepsilon}$ in the chart $(V,Z)$ which covers the whole domain $Q_\delta.$
		For $\varepsilon=0$ and any $\beta$, the vector field  $X_{\beta,\varepsilon}$ has a normally hyperbolic  line of zeros $\{V=0\}$ and then is a slow-fast system 
		of class $C^\infty.$ By adding the trivial equations $\dot{\varepsilon}=0,\dot{\beta}=0$, we can consider $X_{\beta,\varepsilon}$
		as a $\mathcal{C}^\infty$-differential system in four coordinates $(Z,V,\beta,\varepsilon).$  Along the critical surface,   
		$\{V=\varepsilon=0\},$ this $4$-dimensional vector field has at each point $m$ an hyperbolic attracting direction, the $V$-axis with eigenvalue $-1$, and a  $3$-dimensional eigenspace   $E_0(m)$ parallel to $ \{V=0\},$ with eigenvalue $0.$ The slow dynamics is defined along the critical line in the slow time $\tau$ by the equation $\frac{dZ}{d\tau}=-\beta A Z^\alpha.$ 
		This slow dynamics has no singularity as $\beta A>0$ and $Z\geqslant \delta>0.$ 
		\vskip5pt 
		The Fenichel theory (see \cite{F}) shows existence of {\it center manifolds} along normally hyperbolic pieces of the  critical set for general slow-fast differentiable systems.  Here, we will use the formulation of the Fenichel theory as in \cite{DDMR}, which is specific to the smooth slow-fast systems in {\it dimension $2$} and where the existence of smooth center manifolds is proved {\it when the slow dynamics has no singularities.}   
		
		General results of \cite{DDMR} can be translated into the following formulations adapted to our case:
		\begin{theorem} \label{th-Fenichel-1}
			We consider the slow-fast system $X_{\beta,\varepsilon}$ (\ref{eq-XZV}). Let $Z_1$ such that $\delta<Z_1<1$ and $[b_0,b_1] $  an interval in the $\beta$-axis, with $0<b_0<b_1.$ For 
			$\varepsilon_0>0$ small enough, there exists a function 
			$$V=f(Z,\beta, \varepsilon):[\delta,Z_1]\times [b_0,b_1]\times [0,\varepsilon_0]\rightarrow \R,$$ 
			of class $C^\infty,$ with $f(Z,\beta,0)\equiv 0,$ whose graph is an hypersurface $W$ tangent to $X_{\beta,\varepsilon},$ called {\rm center manifold of the system, along the critical set 
				$C=[\delta,Z_1]\times [b_0,b_1]\times \{0\}$.}  
			At each point $m\in C$, this center manifold is tangent to the eigenspace $E_0(m).$ 
			The center manifold is not unique: there exists a $M>0$ (depending on $\delta,Z_1$) such that if  $f_1$ and $f_2$ are   functions defining  two center manifolds, we have that: 
			
			$$| f_1(Z,\beta,\varepsilon)-f_2(Z,\beta,\varepsilon)|\leqslant e^{-\frac{M}{\varepsilon}}$$ 
			for $(Z,\beta,\varepsilon)\in [\delta,Z_1]\times [b_0,b_1]\times (0,\varepsilon_0].$ 
			One says that  the two center manifolds have an exponentially flat contact. A consequence is that there exists a unique  formal center manifold, as all center manifolds along  have the same formal Taylor series along the critical set $C.$ 
		\end{theorem}
		
		All center manifolds attract exponentially the nearby trajectories. More specifically, another result that resulted from \cite{DDMR}, in a form adapted to our case, is the following: 
		\begin{theorem} \label{th-Fenichel-2}
			We consider the slow-fast system $X_{\beta,\varepsilon}$ (\ref{eq-XZV}). Let be $Z_1,\delta, b_0,b_1$  as in Theorem \ref{th-Fenichel-1}.   For
			$\varepsilon_0>0$ small enough, if  
			$$V=f(Z,\beta, \varepsilon):[\delta,Z_1]\times [b_0,b_1]\times [0,\varepsilon_0]\rightarrow \R$$  is the graph of a smooth center manifold along $I$, we have that: 
			$$| f(Z,\beta,\varepsilon)-\tilde\gamma(Z,\beta,\varepsilon)|\leqslant e^{-\frac{M}{\varepsilon}}$$ 
			for $(Z,\beta,\varepsilon)\in [\delta,Z_1]\times [b_0,b_1]\times (0,\varepsilon_0],$   with the same constant $M$ as in Theorem \ref{th-Fenichel-1}.
		\end{theorem}
		
		\vskip5pt\noindent{\bf Step 5: Finite expansion of  center manifold.}
		It is easy to compute the  following finite $\varepsilon$-expansion  of a function $V=f(Z,\beta,\varepsilon)$ defining a central manifold: 
		\begin{lemma}\label{lem-formal-cm}
			Let $V=f(Z,\beta,\varepsilon)$ be the graph of a center manifold of $X_{\beta,\varepsilon}$
			along $C.$ Then,  we have that 
			$$f(Z,\beta,\varepsilon)=\frac{\varepsilon}{2}+\frac{q_0(1-Z)}{4}\varepsilon^3+O(\varepsilon^4),$$
			uniformly in $(Z,\beta)\in[\delta,Z_1]\times [b_0,b_1].$ 
		\end{lemma}
		\vskip5pt\noindent{\bf Proof}
		
		The tangent vector to the graph of $f$ at the point $(f,Z)$ is equal to 
		$\frac{\partial}{\partial Z}+\frac{\partial f}{\partial Z}\frac{\partial}{\partial V}.$ 
		We have to write that this vector is proportional to the vector $X_{\beta,\varepsilon} (f,Z).$ This gives
		\begin{equation}\label{eq-tangent}
			\frac{\partial f}{\partial Z}=\frac{-f+\varepsilon[\frac{1}{2}+q_0(1-Z)f^2]}{-\varepsilon\beta \psi(f) Z^\alpha}.
		\end{equation}
		Consider  the expansion of $f$  at order $1$ in $\varepsilon$: $f(Z,\beta,\varepsilon)=\varepsilon U_1(Z,\beta)+O(\varepsilon^2).$ We have that 
		$\frac{\partial f}{\partial Z}(Z,\beta,\varepsilon)=\varepsilon\frac{\partial U_1}{\partial Z}(Z,\beta)+O(\varepsilon^2).$ Plugging these expansions in (\ref{eq-tangent}) gives that:
		$$\varepsilon\frac{\partial U_1}{\partial Z}+O(\varepsilon^2)=\frac{-U_1+\frac{1}{2}}{-\beta A }+O(\varepsilon).$$
		Taking $\varepsilon=0,$ it comes that $U_1=\frac{1}{2}.$ Then, 
		$f=\frac{\varepsilon}{2}+\varepsilon^2U_2(Z,\beta)+O(\varepsilon^3)$ and 
		$\frac{\partial f}{\partial Z}=\varepsilon^2\frac{\partial U_2}{\partial Z}+O(\varepsilon^3).$
		Using these expansions in (\ref{eq-tangent}),
		$$\varepsilon^2\frac{\partial U_2}{\partial Z}+O(\varepsilon^3)
		=\frac{-\frac{\varepsilon}{2}-\varepsilon^2U_2+O(\varepsilon^3)+\varepsilon(\frac{1}{2}+O(\varepsilon^2))}{-\varepsilon \beta (A+O(\varepsilon)) Z^\alpha}.$$
		This implies that  $U_2=O(\varepsilon).$ Again, taking $\varepsilon= 0,$ we obtain  that $U_2=0.$ 
		Finally, $f(Z,\beta,\varepsilon)$ reads $f(Z,\beta,\varepsilon)=\frac{\varepsilon}{2}+U_3(Z,\beta)\varepsilon^3+O(\varepsilon^4)$ and (\ref{eq-tangent}) yields
		$$\varepsilon^3\frac{\partial U_3}{\partial Z}+O(\varepsilon^4)
		=\frac{-\frac{\varepsilon}{2}-\varepsilon^3U_3+\varepsilon(\frac{1}{2}+\frac{q_0(1-Z)}{4}\varepsilon^2)+O(\varepsilon^4)}{-\varepsilon\beta (A+O(\varepsilon) )Z^\alpha},$$
		which in turn implies that 
		$\displaystyle \varepsilon\frac{\partial U_3}{\partial Z}=\frac{-U_3+\frac{q_0(1-Z)}{4}}{\beta A Z^\alpha}+O(\varepsilon).$ For  $\varepsilon=0$,  $\displaystyle U_3=\frac{q_0(1-Z)}{4}.$
	\end{proof}
	
	\vskip5pt\noindent{\bf Step 6:  Proof of Lemma \ref{lem-rel-position}}. It follows from  Theorem \ref{th-Fenichel-2} and Lemma \ref{lem-formal-cm}
	that :
	$$\tilde \gamma(Z,\beta,\varepsilon)=\frac{\varepsilon}{2}+\frac{q_0(1-Z)}{4}\varepsilon^3+O(\varepsilon^4),$$
	uniformly in $(Z,\beta)\in [\delta,Z_1]\times [b_0,b_1].$ As a consequence we have that $\frac{\partial \tilde \gamma}{\partial Z}(Z,\beta,\varepsilon)<0,$
	for $\varepsilon>0$ small enough, uniformly in  $(Z,\beta)\in [\delta,Z_1]\times [b_0,b_1].$
	On the other side, as the limit vector field for $\varepsilon\rightarrow  0$ is  horizontal, 
	we have that 
	$\frac{\partial \tilde \gamma}{\partial Z}(1,\beta,\varepsilon)\rightarrow +\infty$
	for $\varepsilon\rightarrow 0,$ uniformly in  $\beta\in[b_0,b_1].$ It follows that there exists a value 
	$\bar Z(\beta)\in ]Z_1,1[$ where $\frac{\partial \tilde \gamma}{\partial Z}(\bar Z(\beta),\beta,\varepsilon)=0$ , if $\varepsilon>0$ small enough, uniformly in 
	$\beta\in[b_0,b_1].$ This proves the statement of Lemma  \ref{lem-rel-position}.
	
	\begin{figure}[htp]
		\begin{center}
			\includegraphics[scale=0.5]{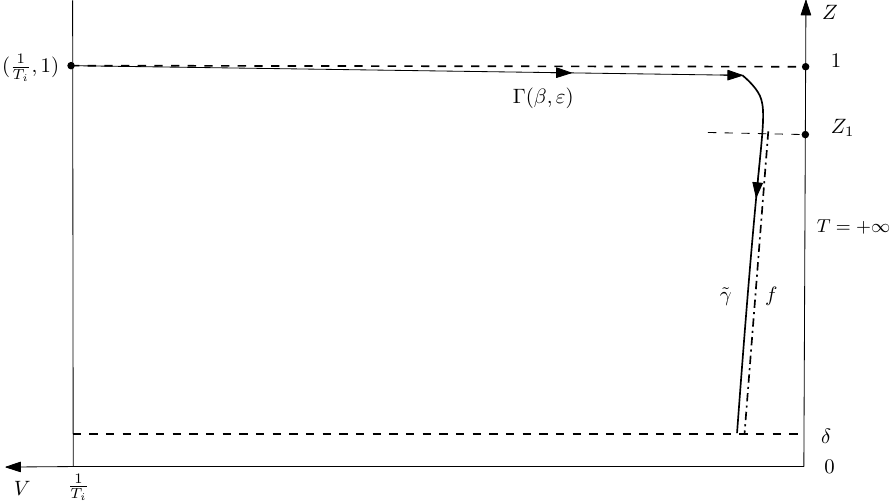}\\
			\caption{$X_{\beta,\varepsilon}$ for $\varepsilon$ small}\label{fig-Vepsilon}
		\end{center}
	\end{figure}

	\section{Numerical results}\label{numerics}
	
	To better understand the Majda-Rosales model with fractional order and confirm our theoretical results, we perform some numerical computations in view of the bifurcation diagram, starting from the vector field \eqref{eq-tildeXU}.

	First, we focus on the Heaviside function $\varphi(T)=H(T-T_{i})$, try to find an equivalent form of the solution of the system \eqref{eq-tildeXU}, and plot the bifurcation diagram using Maple. Then we move on to the Arrhenius law $\varphi(T)=H(T-T_{i})\exp(-\frac{T_{a}}{T})$, using some qualitative results discussed in the previous chapters, we design numerical methods to explore and confirm the specific properties of the respective bifurcation diagrams.
	
	Two methods are implemented: an analytical method (see Appendix \ref{appendixA}) and a shooting method (see Appendix \ref{AppendixB}). For the sake of simplicity, the analytical method will be limited to the Heaviside law in the special case $\alpha = 0.5$. The shooting method will be presented for the Heaviside function and the Arrhenius law, the Heaviside function being clearly a simplification as mentioned earlier. 
	
	\subsection{Analytical method for Heaviside law}
	First, we try to derive an expression that connects the initial condition
	\((U(0), T(0)) = \left(\frac{1}{1-\alpha}, T_i\right)\) with the final condition $(U(-\ell'),T(-\ell'))=(0, T_{0}(c))$ (or $(U(-\ell'),T(-\ell'))=(0,T_{1}(c))$) while characterizing the dependence on the parameters $c$ and $\beta$. After integrating the second ODE in the system \eqref{eq-tildeXU} with the initial condition $U(0)=\frac{1}{1-\alpha}$, we have
	\begin{align}
		U(\tau)=\beta\tau+\frac{1}{1-\alpha}, \qquad 	\ell'=\frac{1}{\beta(1-\alpha)}.
	\end{align}
	This immediately decouples the system as a Riccati equation
	\begin{align}\label{Riccati}
		{\dot T}=-cT+\frac{1}{2}T^2+q_0\Big(1-(1-\alpha)^{\frac{1}{1-\alpha}}(\beta\tau+\frac{1}{1-\alpha})^{\frac{1}{1-\alpha}}\Big),
	\end{align}
	with initial condition $T(0)=T_{i}$. 
	
	Several types of solutions can be observed by applying the Runge-Kutta 4/5 method to \eqref{Riccati}  for different values of the parameter $\beta>0$ and a fixed value of the parameter $c>c^{cj}$, see Figures \ref{pic-1} to \ref{pic-3}. 
	\begin{figure}[htp]
		\centering
		\begin{minipage}[t]{0.49\linewidth}
			\centering
			\includegraphics[width=1.1\linewidth]{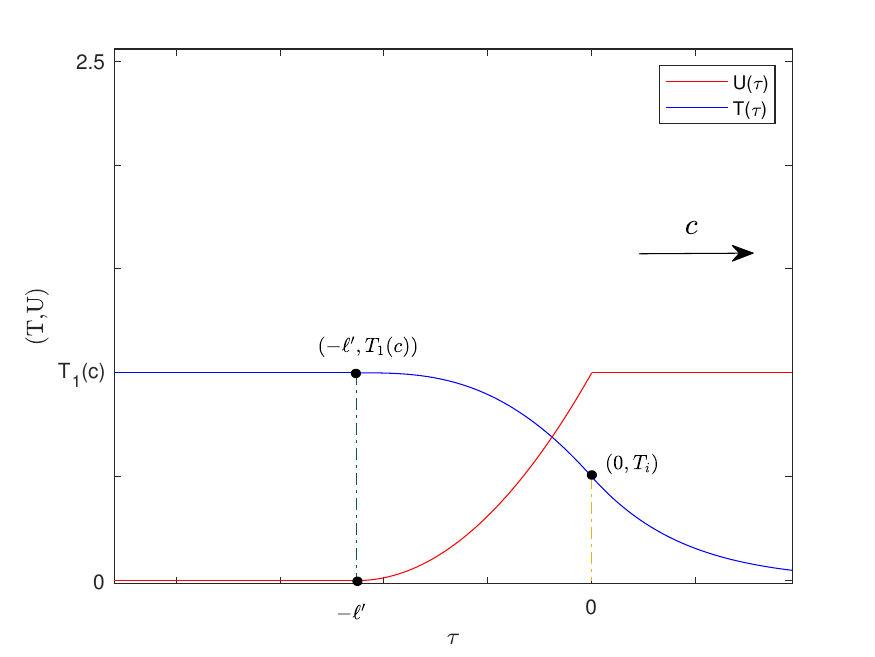}
			\label{compact2.pdf}
			\subcaption*{(a)}	
		\end{minipage}
		\begin{minipage}[t]{0.49\linewidth}
			\centering
			\includegraphics[width=1.13\linewidth]{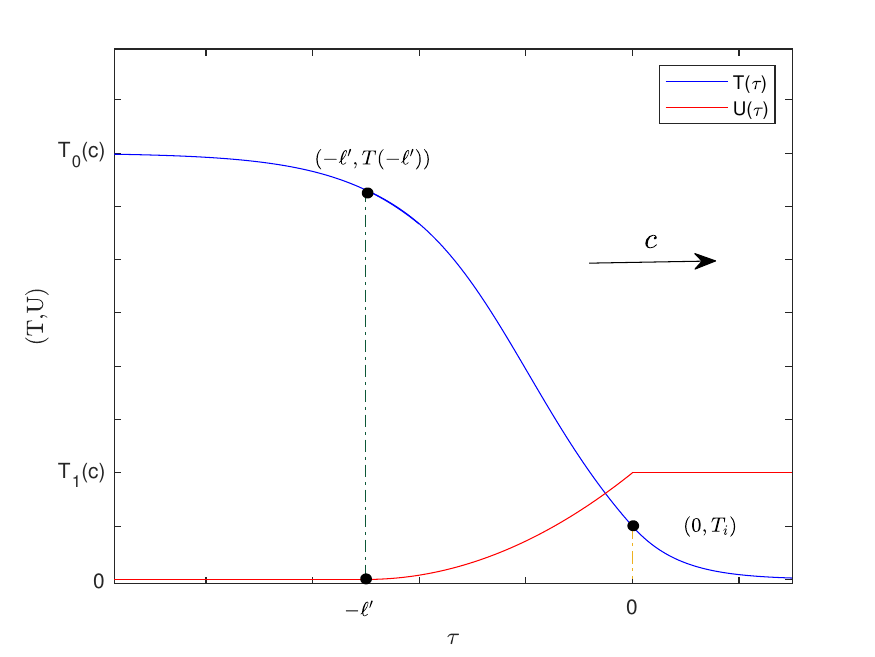}
			\label{monotonic}
			\subcaption*{(b)}
		\end{minipage}
		\caption{Profiles of solutions of the vector field \eqref{eq-tildeXU_bis}, extended to the entire real line, with values of parameters $\alpha=0.5$, $q_{0}=2$, $T_{i}=0.5$, $c=2.5$. $(a)$ $\beta=1.7675$, weak reactive waves of special type; $(b)$ $\beta=0.8$, strong reactive wave with monotonic type with velocity $c$. The ignition temperature is reached at $\tau=0$ and the free interface is at $\tau=-\ell'$.}
		\label{pic-1}
	\end{figure}
	\begin{figure}[htp]
		\label{type2}
		\begin{minipage}[t]{0.49\linewidth}
			\centering
			\includegraphics[width=1.1\linewidth]{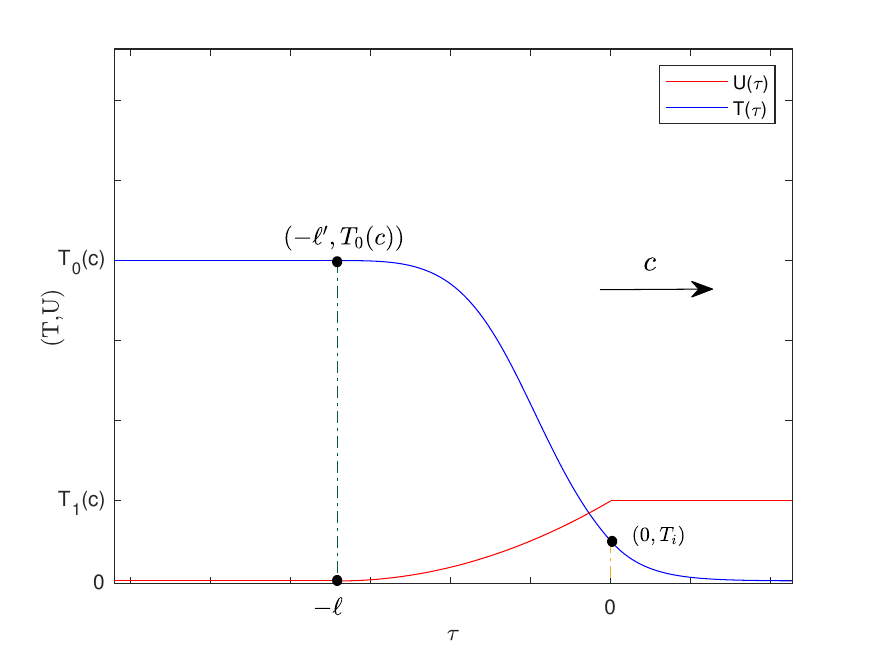}
			\label{compact1}
			\subcaption*{(c)}
		\end{minipage}
		\begin{minipage}[t]{0.49\linewidth}
			\centering
			\includegraphics[width=1.1\linewidth]{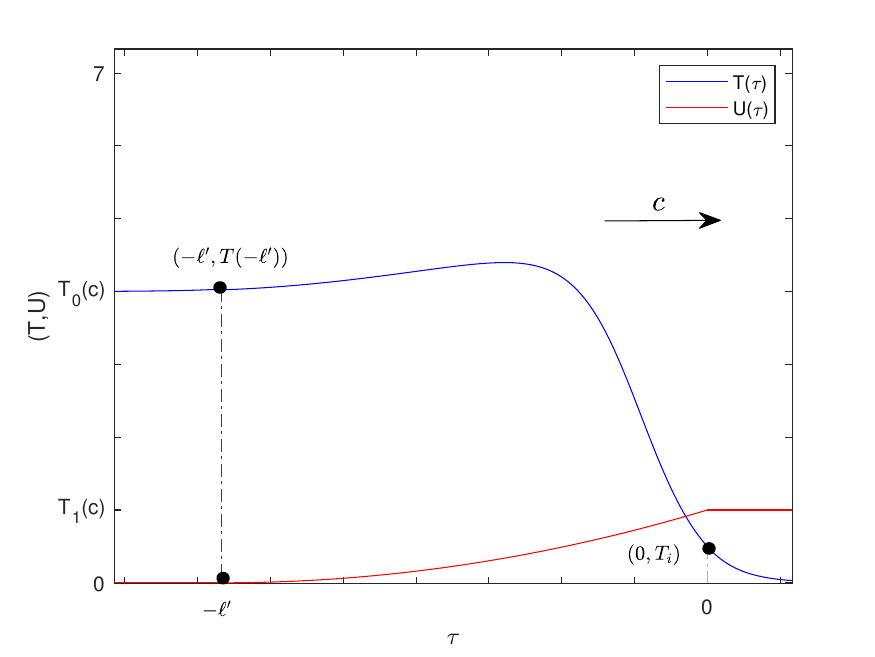}
			\label{bump.pdf}
			\subcaption*{(d)}
		\end{minipage}
		\caption{Profiles of solutions of the vector field \eqref{eq-tildeXU_bis}, extended to the entire real line, with values of parameters $\alpha=0.5$, $q_{0}=2$, $T_{i}=0.5$, $c=2.5$.  $(c)$ $\beta=0.582$, strong reactive wave of special type; $(d)$ $\beta=0.4$, strong reactive wave with a bump.}
		\label{pic-2}
	\end{figure}
	\begin{figure}[htp]
		\label{type3}
		\begin{minipage}[t]{0.49\linewidth}
			\centering
			\includegraphics[width=1.1\linewidth]{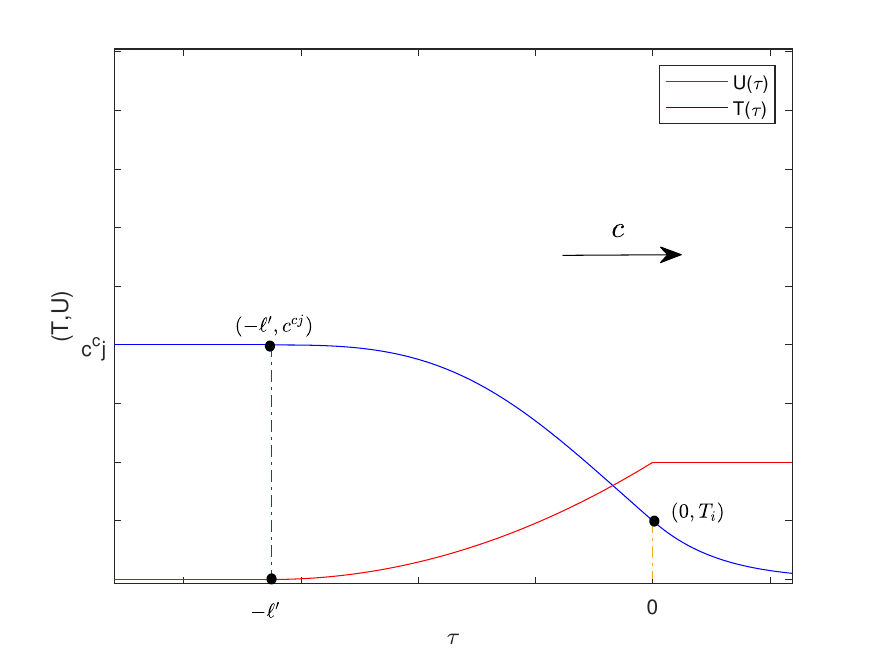}
			\label{cj1}
			\subcaption*{(e)}
		\end{minipage}
		\begin{minipage}[t]{0.49\linewidth}
			\centering
			\includegraphics[width=1.1\linewidth]{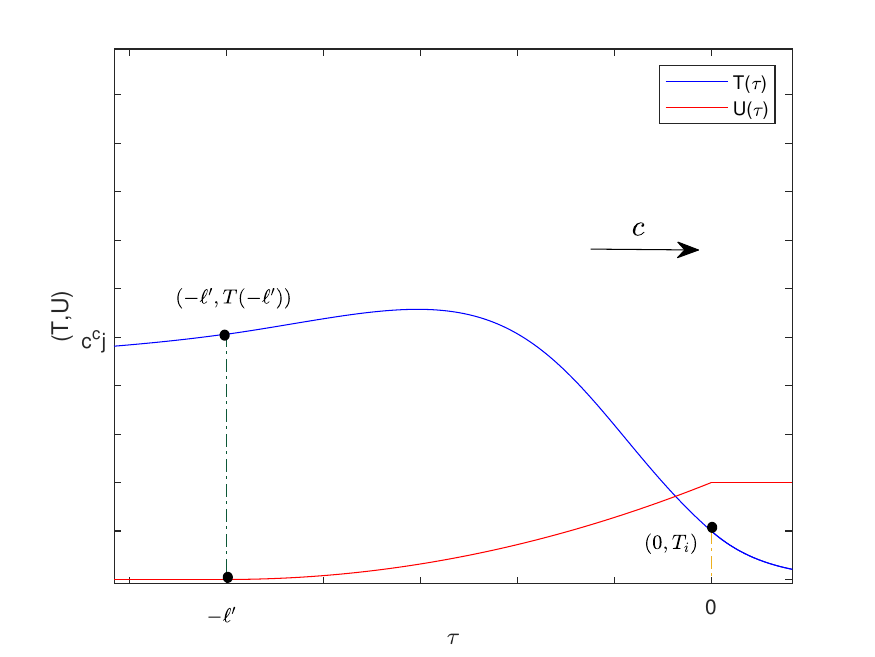}
			\label{cj2}
			\subcaption*{(f)}
		\end{minipage}
		\caption{Profiles of solutions of the vector field \eqref{eq-tildeXU_bis}, extended to the entire real line, with values of parameters $\alpha=0.5$, $q_{0}=2$, $T_{i}=0.5$, $c=2$. $(e)$ $\beta=0.615$, CJ reactive wave of special type. $(f)$ $\beta=0.4$, CJ reactive wave with a bump (combustion spike).}
		\label{pic-3}
	\end{figure}

	As announced, for simplicity we fix \(\alpha = 0.5\), $q_{0}=2$, and $T_{i}=0.5$ in this section. Then, we use
	Maple's ODE solver $\mathtt{dsolve}$ to derive an equivalent form of the solution for \eqref{Riccati} with the initial condition $T(0)=T_{i}$ and then derive an equation $T(-\ell')-T_{0}(c)=0$ (see \eqref{eq-implicit beta0}) by plugging the final conditions into the latter. Finally, using Maple's implicit function tool  $\mathtt{implicitplot}$ (see \cite{HM}) on the formula \eqref{eq-Riccati solution}), with the range $\mathtt{\beta = 0.01 . 5}$, $\mathtt{c = 2 .. 5}$ and the grid option $\mathtt{gridrefine = 5}$, we generate a visual representation of the local implicit function for the curve $\beta_{0}(c)$ in the parameter space. Similarly, the result of the equation $T(-\ell')-T_{1}(c)=0$ is displayed in \eqref{eq-implicit beta1}. Finally repeating the previous steps, the bifurcation diagram is displayed as Figure \ref{Bifurcation maple}.
	
	\begin{figure}[h]
		\centering
		\includegraphics[width=0.5\linewidth]{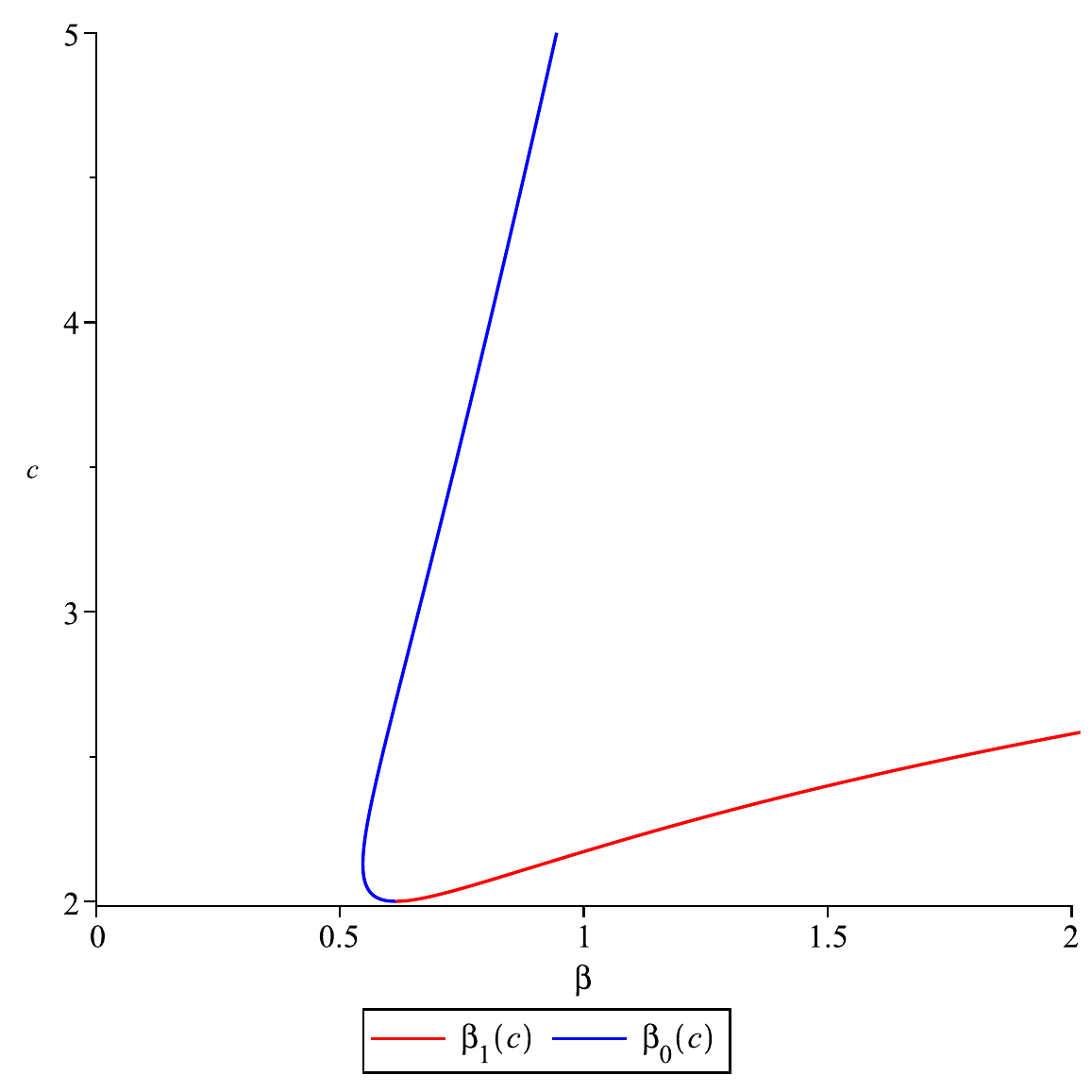}
		\caption{Bifurcation diagram in the case of Heaviside function with parameter values $\alpha=0.5$, $q_{0}=2$, $T_{i}=0.5$.}
		\label{Bifurcation maple}
	\end{figure}
	\begin{figure}
		\centering
		\includegraphics[width=0.5\linewidth]{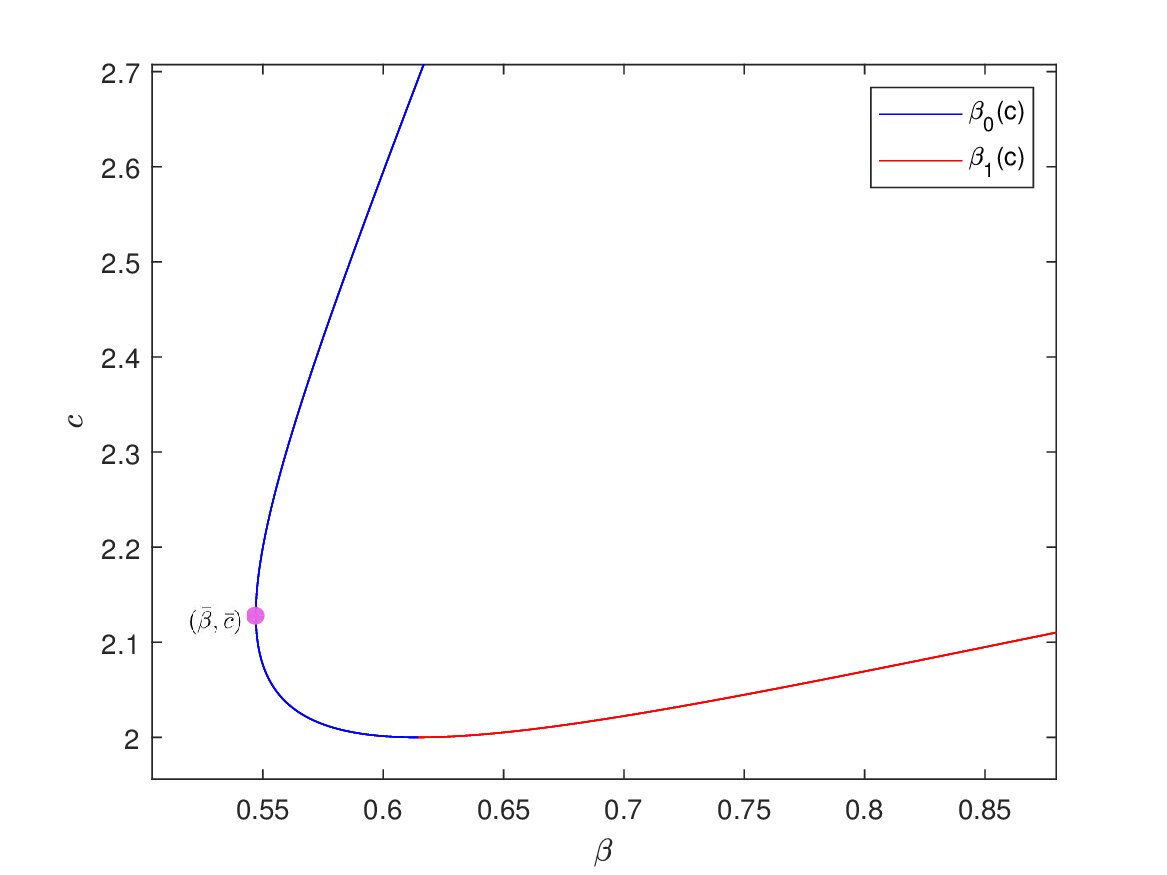}
		\caption{Enlargement of Fig. \ref{Bifurcation maple} around the turning point $(\bar{\beta},{\bar{c}})$.}
		\label{Bifurcation maple1}
	\end{figure}	
	\begin{figure}[h]
		\centering
		\includegraphics[width=0.5\linewidth]{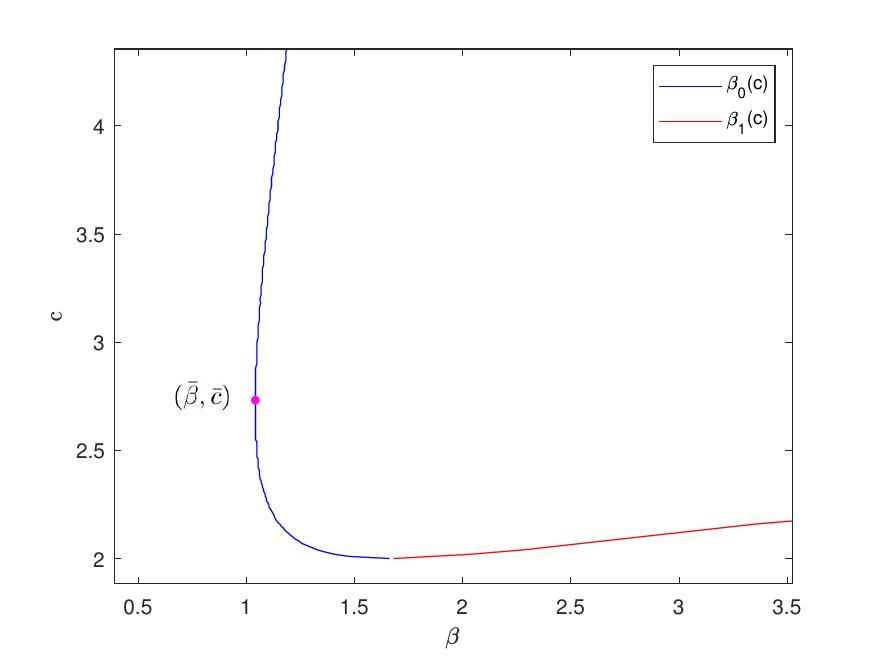}
		\caption{Enlargement of Fig. \ref{Arrhenius_bifurcation} in a neighborhood of the turning point.}
		\label{Arrhenius_bifurcation1}
	\end{figure}
	We observe that the bifurcation diagram generated by Maple agrees with the qualitative results:
	\begin{enumerate}[label=(\roman*)]
		\item the curves $\beta_0(c)$ and $\beta_{1}(c)$ are smooth, they merge at the critical point $(\beta^{\mathit{cj}},\CJ)$ of the curve $(B)$, see Proposition \ref{prop-global}; 
		\item we also observe a local extremum of the function $\beta_{0}(c)$ for $c>c^{cj}$, see Figure \ref{Bifurcation maple1}, namely a turning point $(\bar{\beta},\bar{c})$ which seemingly is unique. 
	\end{enumerate}
	
	\begin{remark}
		The accuracy of the computation of the numerical value of the critical point $(\beta^{cj},c^{cj})$ may be improved as follows:
		plugging $c=c^{cj}$ into $\eqref{eq-implicit beta0}$ yields the formula:
		%To improve the accuracy of the numerical value of the critical point $(\beta^{cj},c^{cj})$, plugging $c=c^{cj}$ into the formula $\eqref{eq-implicit beta0}$, we derive a general formula:
		\begin{align}\label{eq-beta0c-error}
			T(-\ell^{'})-c^{cj}=\frac{3 \beta \Gamma \! \left(\frac{3}{4}\right)^{2} \sqrt{\frac{1}{\beta}}\, \left( \frac{3}{4} I_{-\frac{1}{4}}\! \left(\frac{1}{\beta}\right)-I_{\frac{3}{4}}\! \left(\frac{1}{\beta}\right)\right)}{\pi  \left( \frac{9}{8} I_{\frac{1}{4}}\! \left(\frac{1}{\beta}\right)-\frac{3}{2} I_{-\frac{3}{4}}\left(\frac{1}{\beta}\right)\right)}=0,
		\end{align}
		where $I_{v}(z)$ is a Bessel function of the first kind, $\Gamma(u)$ the Gamma function. Using Maple, formula \eqref{eq-beta0c-error} provides the following approximated value: $\beta^{\mathit{cj}} \approx 0.614452673918923$. 
		%then substituting this estimate into \eqref{eq-beta0c-error}, we get $T(-\ell^{'})-c^{cj}\approx-4.80\times10^{-14}$.
	\end{remark}

	\subsection{Numerical shooting method for Arrhenius kinetics}\label{sec-Arrhenius kinetics}
	
	A number of problems arise where the position of the free interface cannot be determined due to the inability to decouple the system when considering the Arrhenius function. To numerically simulate the bifurcation curves $\beta_0(c)$ and $\beta_1(c)$, we employ a shooting-based numerical scheme. Specifically, we discretize the parameter space with uniform meshes, and for each fixed parameter $c$, we determine the corresponding $\beta$ value by exploiting the monotonicity of the transition mappings(see proposition \ref{prop-rotating}) , combined with a highly accurate Runge-Kutta method for solving the ODE system (see, e.g., \cite{Keller, Doedel, Reichelt}). The detailed numerical procedure is presented in Appendix \ref{AppendixB}, following the implementation of the numerical method. The bifurcation diagram is shown in Figure \ref{Arrhenius_bifurcation}.
	
	Finally, we study the influence of the parameter $\alpha$ on the curves $\beta_0(c)$ and $\beta_1(c)$ in Figure \ref{fig-diffalpha}.

	\begin{figure}[h]
		\centering
		\includegraphics[width=0.5\linewidth]{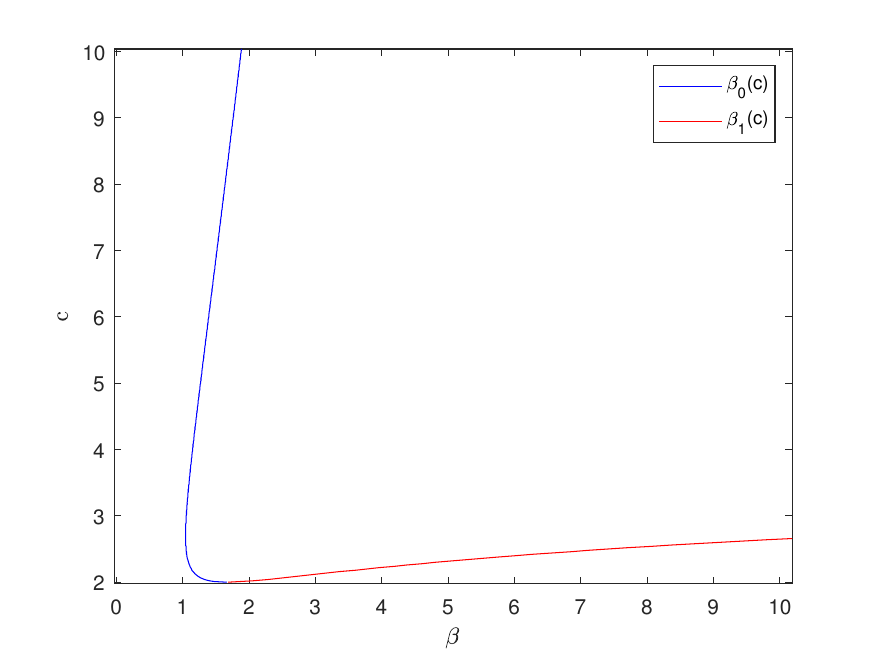}
		\caption{Bifurcation diagram for Arrhenius kinetics with parameter values: $\alpha=0.5$, $T_{a}=1$, $q_{0}=2$, $T_{i}=0.5$.}
		\label{Arrhenius_bifurcation}
	\end{figure}

	\subsection{Validation}
	The validation method employed is predicated on the observation that the numerical shooting is evidently applicable to the Heaviside function.
	
	%The numerical results are validated against known analytical solutions where applicable, ensuring the robustness of our computational approach. 
	In order to compare the results obtained by the two methods, we substitute \(\bar{\ell}'\) with \(\ell'=\frac{1}{\beta(1-\alpha)}\) in step (4) of Appendix \ref{AppendixB}. %This substitution yields the results displayed in Figure \ref{Bifurcation maple}. 
	Figure \ref{Comparison of two methods} showcases our results which exhibit a nearly complete overlap of the two curves produced by the two methods, which therefore give similar results whenever they are applicable. 
	
	An other observation is the consistent behavior of the bifurcation curves $\beta_0(c)$ and $\beta_1(c)$ across different modeling scenarios (Heaviside and Arrhenius). This similarity suggests that the underlying dynamics of the system is robust to variations in the functional form of the reaction rate.
	
	Finally, we point out the flexibility of the shooting method which allows for easy adjustments to the parameter intervals. This makes it easy to thoroughly explore the bifurcation landscape. 
	Future research will focus on the influence of additional parameters on the behavior of the system, extending the investigation into the parameter space to uncover more complex dynamic features.
	
	\begin{figure}
		\centering
		\includegraphics[width=0.5\linewidth]{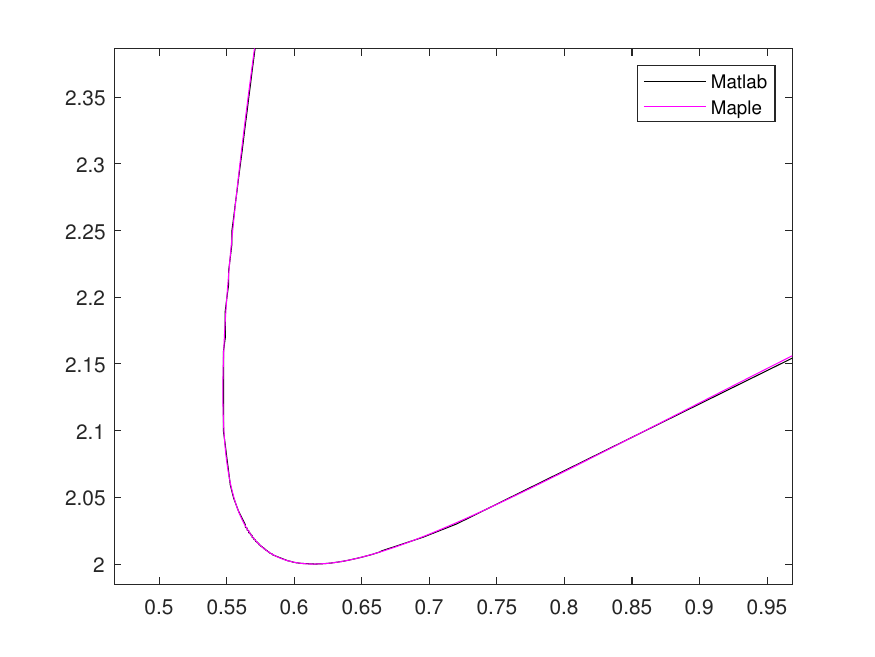}
		\caption{Comparison of the two methods.}
	%	\ref{Comparison of two methods} 
		\label{Comparison of two methods}
	\end{figure}
	
	\begin{figure}
		\centering
		\includegraphics[width=0.5\linewidth]{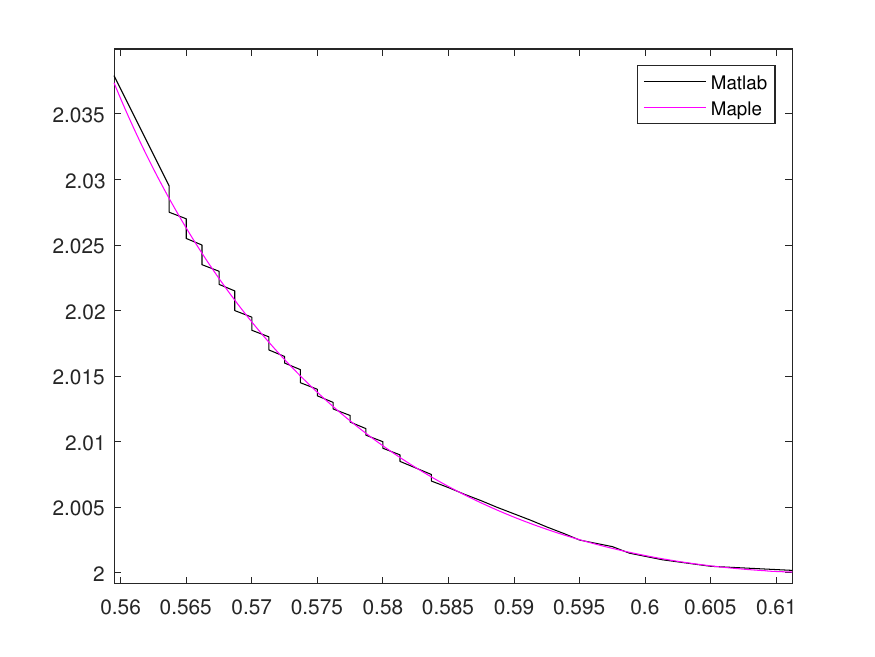}
		\caption{Comparison of the two methods (enlargement).}
	%	\ref{Comparison of two methods(zoom)} 
		\label{Comparison of two methods(zoom)}
	\end{figure}

	%\subsection{Dependency of the bifurcation diagram on the parameter $\alpha$}
	%In the previous sections, we established some numerical means to compute the simple curve $B$ in parameter space, and it has bifurcation behavior around $c=c^{cj}$ according to the numerical simulations. A natural question is to study the dependence of the bifurcation diagram on the parameters; here, we study the influence of the different values of $\alpha$ on the different branches of the curve $B$ for the Heaviside function.
	
	%After performing the calculations with Maple, we found that for the different values of the parameter $\alpha$, some equations cannot be expressed by Maple in an equivalent way. Therefore, to show the changes of the curves with the variation of the parameters, we will use the method described in Section \ref{sec-Arrhenius kinetics}. The details will not be discussed here, but the results will be given. 
	\begin{figure}[h]
		\centering
		\begin{minipage}[t]{0.49\linewidth}
			\centering
			\includegraphics[width=0.9\linewidth]{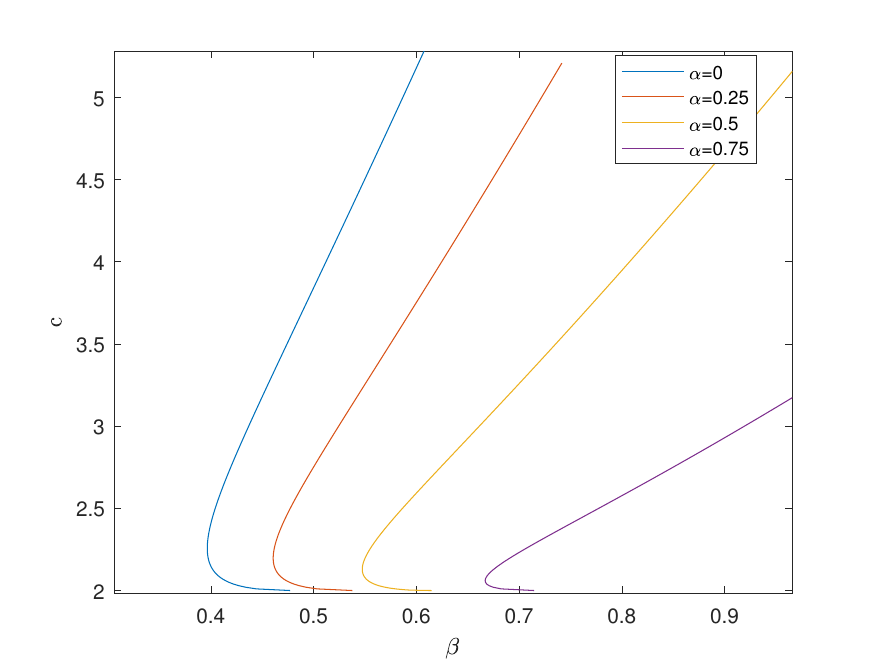}
		\end{minipage}
		\begin{minipage}[t]{0.49\linewidth}
			\centering
			\includegraphics[width=0.9\linewidth]{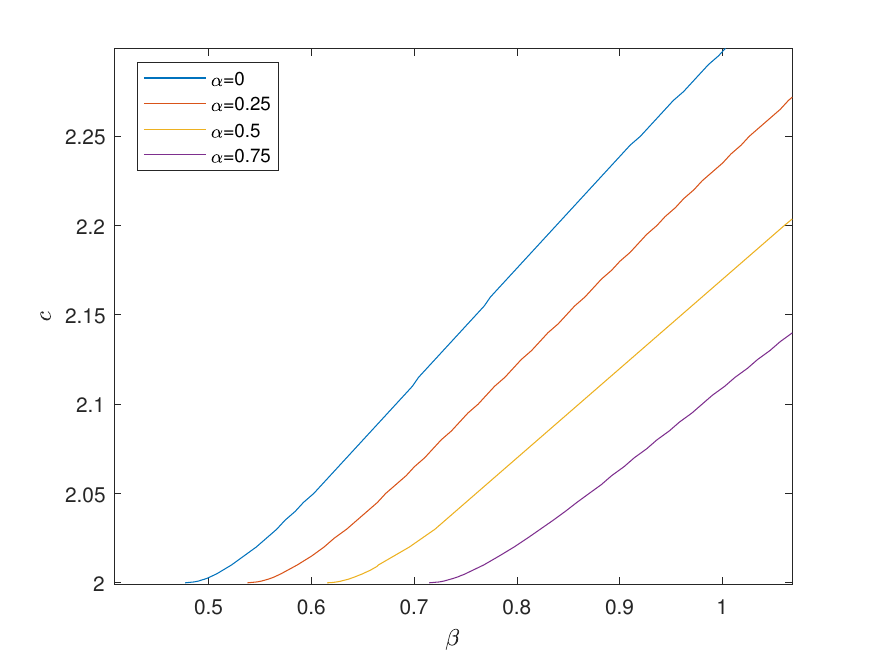}
		\end{minipage}
		\caption{Effect of parameter $\alpha$ on the curve $\beta_{0}(c)$ (left) and on the curve $\beta_{1}(c)$ (right). Here, $T_{a}=1$, $q_{0}=2$, $T_{i}=0.5$.}
		\label{fig-diffalpha}
	\end{figure}
	\newpage
	
	\setcounter{section}{0}
	\setcounter{theorem}{0}
	\setcounter{figure}{0}
	%%%%%%%%%%%%%%%%%%%%%%%%%%%%%%
	\numberwithin{theorem}{section}

	\numberwithin{figure}{section}

	\section{Acknowledgments}
	
	The authors are appreciative of the fruitful discussions with Prof. Victor Roytburd. They also wish to thank Prof. Hanchun Yang for bringing reference \cite{Li92} to their attention and providing sound advice. C.-M. B. greatly acknowledges Prof. Mei Ming and the Department of Mathematics and Statistics of Yunnan University for their warm hospitality during his visiting positions in 2023 and 2024. This project was supported by the National Natural Science Foundation of China (No.12222116). The Institut de Mathématiques de Bourgogne receives support from EHPHI Graduate School (contract ANR-17-0002).

	%%%%%%%%%%%%%%%%%%%%%%%%%%%%%%%

	\appendix
	\section{Solving Riccati equation  ($\alpha=0.5$).}\label{appendixA}
	This first appendix shows the results of the calculation of the Riccati equation \eqref{Riccati} when the special value $\alpha = 0.5$ is considered.

	We recall that the Kummer's function $M(a;b;z)$ is a generalized hypergeometric series, often referred to as the confluent hypergeometric function of the first kind, given by 
	\begin{equation}
		M(a;b;z)= _{1}\!\!F_{1}(a,b,z) = \sum_{n=0}^{\infty}\frac{(a)_n z^n}{(b)_n n!},
	\end{equation}
	where $(a)_n$ is the Pochhammer symbol (see, e.g. \cite{Mathews22}).
	
	The equivalent form of the solution of the Riccati equation with the initial condition $T(0)=T_{i}$ is displayed as follows:
	\begin{align}\label{eq-Riccati solution}
		T(\tau)=&\left[-3 \left(\beta \tau +2\right)^{2} \left(\beta {} ^{}{ {  {M^{}}}}\! \left(\frac{2 \beta \sqrt{q_{0}}+\sqrt{2}\, \left(c^{2}-2 q_{0} \right)}{8 \beta \sqrt{q_{0}}};\frac{1}{2};\frac{\sqrt{2}\, \sqrt{q_{0}}}{\beta}\right)\left(-\frac{2 \sqrt{2}\, q_{0}^{\frac{3}{2}}}{3}\right.\right.\right.\notag\\
		&\left.\left.+\left(\left(c -T_{i} \right) \beta +\frac{c^{2}}{3}\right) \sqrt{2}\, \sqrt{q_{0}}+2 q_{0} \beta +\frac{\left(c -T_{i} \right) \left(c^{2}-2 q_{0} \right)}{3}\right)-2 \left(-\frac{4 \beta \sqrt{2}\, q_{0}^{\frac{3}{2}}}{3}+\frac{2 \beta \,c^{2} \sqrt{2}\, \sqrt{q_{0}}}{3}\right.\right.\notag\\			
		&\left.\left.+q_{0} \,\beta^{2}+\frac{\left(c^{2}-2 q_{0} \right)^{2}}{6}\right) {} ^{}{ {  {M^{}}}}\! \left(\frac{10 \beta \sqrt{q_{0}}+\sqrt{2}\, \left(c^{2}-2 q_{0} \right)}{8 \beta \sqrt{q_{0}}};\frac{3}{2};\frac{\sqrt{2}\, \sqrt{q_{0}}}{\beta}\right)\right) {} ^{}\notag\\			
		&\ \ { {  {M^{}}}}\! \left(\frac{14 \beta \sqrt{q_{0}}+\sqrt{2}\, \left(c^{2}-2 q_{0} \right)}{8 \beta \sqrt{q_{0}}};\frac{5}{2};\frac{\sqrt{2}\, \sqrt{q_{0}}\, \left(\beta \tau +2\right)^{2}}{4 \beta}\right)+3 \left(\beta \left(\left(\tau^{2} \left(c -T_{i} \right) \beta^{2}+\left(-4\right.\right.\right.\right.\notag\\			
		&\left.\left.\left.\left.+\left(6 c -4 T_{i} \right) \tau \right) \beta +8 c -4 T_{i} \right) \sqrt{2}\, \sqrt{q_{0}}+2 q_{0} \,\beta^{2} \tau^{2}+\left(\left(-2 c T_{i} +2 c^{2}+8 q_{0} \right) \tau -4 c +4 T_{i} \right) \beta\right.\right.\notag\\
		&\left.\left.+4 c^{2}-4 c T_{i} +8 q_{0} \right) {} ^{}{ {  {M^{}}}}\! \left(\frac{2 \beta \sqrt{q_{0}}+\sqrt{2}\, \left(c^{2}-2 q_{0} \right)}{8 \beta \sqrt{q_{0}}};\frac{1}{2};\frac{\sqrt{2}\, \sqrt{q_{0}}}{\beta}\right)\right.\notag\\
		&\left.-2 \left(-\sqrt{2}\, \left(\beta \tau +2\right)^{2} q_{0}^{\frac{3}{2}}+\frac{\left(\left(c^{2} \tau^{2}+2 c \tau -4\right) \beta^{2}+\left(4 c^{2} \tau +4 c \right) \beta +4 c^{2}\right) \sqrt{2}\, \sqrt{q_{0}}}{2}+\beta^{3} q_{0} \,\tau^{2}\right.\right.\notag\\
		&\left.\left.+4 q_{0} \,\beta^{2} \tau +\left(c \left(c^{2}-2 q_{0} \right) \tau -2 c^{2}+8 q_{0} \right) \beta +2 c^{3}-4 c q_{0} \right)\right.\notag\\
		&\left.{} ^{}{ {  {M^{}}}}\! \left(\frac{10 \beta \sqrt{q_{0}}+\sqrt{2}\, \left(c^{2}-2 q_{0} \right)}{8 \beta \sqrt{q_{0}}};\frac{3}{2};\frac{\sqrt{2}\, \sqrt{q_{0}}}{\beta}\right)\right)\notag\\ 
		&\beta {} ^{}{ {  {M^{}}}}\! \left(\frac{6 \beta \sqrt{q_{0}}+\sqrt{2}\, \left(c^{2}-2 q_{0} \right)}{8 \beta \sqrt{q_{0}}};\frac{3}{2};\frac{\sqrt{2}\, \sqrt{q_{0}}\, \left(\beta \tau +2\right)^{2}}{4 \beta}\right)-6 \left(\beta \tau +2\right) \left(\beta \left(2 \sqrt{2}\, q_{0}^{\frac{3}{2}}\right.\right.\notag\\	
		&\left.\left.+\left(\beta^{2}+\left(-c +T_{i} \right) \beta -c^{2}\right) \sqrt{2}\, \sqrt{q_{0}}+\left(c^{2}-4 q_{0} \right) \beta -\left(c -T_{i} \right) \left(c^{2}-2 q_{0} \right)\right)\right.\notag\\
		&\left.{} ^{}{ {  {M^{}}}}\! \left(\frac{6 \beta \sqrt{q_{0}}+\sqrt{2}\, \left(c^{2}-2 q_{0} \right)}{8 \beta \sqrt{q_{0}}};\frac{3}{2};\frac{\sqrt{2}\, \sqrt{q_{0}}}{\beta}\right)+2 \left(-\frac{4 \beta \sqrt{2}\, q_{0}^{\frac{3}{2}}}{3}+\frac{2 \beta \,c^{2} \sqrt{2}\, \sqrt{q_{0}}}{3}+q_{0} \,\beta^{2}\right.\right.\notag\\
		&\left.\left.+\frac{\left(c^{2}-2 q_{0} \right)^{2}}{6}\right) {} ^{}{ {  {M^{}}}}\! \left(\frac{14 \beta \sqrt{q_{0}}+\sqrt{2}\, \left(c^{2}-2 q_{0} \right)}{8 \beta \sqrt{q_{0}}};\frac{5}{2};\frac{\sqrt{2}\, \sqrt{q_{0}}}{\beta}\right)\right) {} ^{}\notag\\
		&{ {  {M^{}}}}\! \left(\frac{10 \beta \sqrt{q_{0}}+\sqrt{2}\, \left(c^{2}-2 q_{0} \right)}{8 \beta \sqrt{q_{0}}};\frac{3}{2};\frac{\sqrt{2}\, \sqrt{q_{0}}\, \left(\beta \tau +2\right)^{2}}{4 \beta}\right)+6 \left(\beta \left(\sqrt{2}\, \left(\tau \,\beta^{2}+\left(2+\left(-c\right.\right.\right.\right.\right.\notag\\
		&\left.\left.\left.\left.\left.+T_{i} \right) \tau \right) \beta -4 c +2 T_{i} \right) \sqrt{q_{0}}+\left(-2 q_{0} \tau +2 c \right) \beta -2 c^{2}+2 c T_{i} -4 q_{0} \right) {} ^{}\right.\notag\\
		&\left.{ {  {M^{}}}}\! \left(\frac{6 \beta \sqrt{q_{0}}+\sqrt{2}\, \left(c^{2}-2 q_{0} \right)}{8 \beta \sqrt{q_{0}}};\frac{3}{2};\frac{\sqrt{2}\, \sqrt{q_{0}}}{\beta}\right)+2 \left(-\frac{\sqrt{2}\, \left(\beta \tau +2\right) q_{0}^{\frac{3}{2}}}{3}\right.\right.\notag\\
		&\left.\left.+\frac{c \sqrt{2}\, \left(\beta c \tau+6 \beta +2 c \right) \sqrt{q_{0}}}{6}+q_{0} \,\beta^{2} \tau +2 q_{0} \beta +\frac{c \left(c^{2}-2 q_{0} \right)}{3}\right) {} ^{}\right.\notag\\			
		&\left.{ {  {M^{}}}}\! \left(\frac{14 \beta \sqrt{q_{0}}+\sqrt{2}\, \left(c^{2}-2 q_{0} \right)}{8 \beta \sqrt{q_{0}}};\frac{5}{2};\frac{\sqrt{2}\, \sqrt{q_{0}}}{\beta}\right)\right) \beta\notag\\
		&\left.{} ^{}{ {  {M^{}}}}\! \left(\frac{2 \beta \sqrt{q_{0}}+\sqrt{2}\, \left(c^{2}-2 q_{0} \right)}{8 \beta \sqrt{q_{0}}};\frac{1}{2};\frac{\sqrt{2}\, \sqrt{q_{0}}\, \left(\beta \tau +2\right)^{2}}{4 \beta}\right)\right]\notag\\
		&\left[6 \left(\left(\beta \tau +2\right) \left(\beta {}  ^{}{ {  {M^{}}}}\! \left(\frac{2 \beta \sqrt{q_{0}}+\sqrt{2}\, \left(c^{2}-2 q_{0} \right)}{8 \beta \sqrt{q_{0}}};\frac{1}{2};\frac{\sqrt{2}\, \sqrt{q_{0}}}{\beta}\right) \left(\sqrt{2}\, \sqrt{q_{0}}+c -T_{i} \right)\right.\right.\right.\notag\\
		&\left.\left.\left.-\left(\beta \sqrt{2}\, \sqrt{q_{0}}+c^{2}-2 q_{0} \right) {}  ^{}{ {  {M^{}}}}\! \left(\frac{10 \beta \sqrt{q_{0}}+\sqrt{2}\, \left(c^{2}-2 q_{0} \right)}{8 \beta \sqrt{q_{0}}};\frac{3}{2};\frac{\sqrt{2}\, \sqrt{q_{0}}}{\beta}\right)\right) {}  ^{}\right.\right.\notag\\
		&\left.\left.{ {  {M^{}}}}\! \left(\frac{6 \beta \sqrt{q_{0}}+\sqrt{2}\, \left(c^{2}-2 q_{0} \right)}{8 \beta \sqrt{q_{0}}};\frac{3}{2};\frac{\sqrt{2}\, \sqrt{q_{0}}\, \left(\beta \tau +2\right)^{2}}{4 \beta}\right)+2 \left(\beta \left(-\sqrt{2}\, \sqrt{q_{0}}+\beta -c\right.\right.\right.\right. \notag\\
		&\left.\left.\left.\left.+T_{i} \right) {}  ^{}{ {  {M^{}}}}\! \left(\frac{6 \beta \sqrt{q_{0}}+\sqrt{2}\, \left(c^{2}-2 q_{0} \right)}{8 \beta \sqrt{q_{0}}};\frac{3}{2};\frac{\sqrt{2}\, \sqrt{q_{0}}}{\beta}\right)+{}  ^{}\right.\right.\right.\notag\\
		&\left.\left.\left.{ {  {M^{}}}}\! \left(\frac{14 \beta \sqrt{q_{0}}+\sqrt{2}\, \left(c^{2}-2 q_{0} \right)}{8 \beta \sqrt{q_{0}}};\frac{5}{2};\frac{\sqrt{2}\, \sqrt{q_{0}}}{\beta}\right) \left(\beta \sqrt{2}\, \sqrt{q_{0}}+\frac{c^{2}}{3}-\frac{2 q_{0}}{3}\right)\right) {}  ^{}\right.\right.\notag\\
		&\left.\left.{ {  {M^{}}}}\! \left(\frac{2 \beta \sqrt{q_{0}}+\sqrt{2}\, \left(c^{2}-2 q_{0} \right)}{8 \beta \sqrt{q_{0}}};\frac{1}{2};\frac{\sqrt{2}\, \sqrt{q_{0}}\, \left(\beta \tau +2\right)^{2}}{4 \beta}\right)\right) \beta\right]^{-1}.
	\end{align}
	Then the functions determining the implicit function curves $\beta_{0}(c)$ and $\beta_{1}(c)$ are read as $T(-\ell')-T_{0}(c)=0$ and $T(-\ell')-T_{1}(c)=0$, respectively:
	\begin{align}\label{eq-implicit beta0}
		T(-\ell')-T_{0}(c)=&\left[-3 \beta \sqrt{c^{2}-2 q_{0}}\, \left(-\sqrt{2}\, \sqrt{q_{0}}+\beta -c + T_{i} \right)\right.\notag\\ 
		&\left.M \left(\frac{6 \beta \sqrt{q_{0}}+\left(c^{2}-2 q_{0} \right) \sqrt{2}}{8 \beta \sqrt{q_{0}}};\frac{3}{2};\frac{\sqrt{2}\, \sqrt{q_{0}}}{\beta}\right)-3 \left(\beta \sqrt{2}\, \sqrt{q_{0}}+\frac{c^{2}}{3}-\frac{2 q_{0}}{3}\right)\right.\notag\\
		&M \left(\frac{14 \beta \sqrt{q_{0}}+\left(c^{2}-2 q_{0} \right) \sqrt{2}}{8 \beta \sqrt{q_{0}}};\frac{5}{2};\frac{\sqrt{2}\, \sqrt{q_{0}}}{\beta}\right)\sqrt{c^{2}-2 q_{0}}+3 \beta \left(-\beta \left(\sqrt{2}\, \sqrt{q_{0}}\right.\right.\notag\\
		&\left.\left.+c-T_{i}\right)M \left(\frac{2 \beta \sqrt{q_{0}}+\left(c^{2}-2 q_{0} \right) \sqrt{2}}{8 \beta \sqrt{q_{0}}};\frac{1}{2};\frac{\sqrt{2}\, \sqrt{q_{0}}}{\beta}\right)\right.\notag\\
		&\left.\left.+M \left(\frac{10 \beta \sqrt{q_{0}}+\left(c^{2}-2 q_{0} \right) \sqrt{2}}{8 \beta \sqrt{q_{0}}};\frac{3}{2};\frac{\sqrt{2}\, \sqrt{q_{0}}}{\beta}\right) \left(\beta \sqrt{2}\, \sqrt{q_{0}}+c^{2}-2 q_{0} \right)\right)\right]\notag\\
		&\left[3 \beta \left(-\sqrt{2}\, \sqrt{q_{0}}+\beta -c +T_{i} \right) M \left(\frac{6 \beta \sqrt{q_{0}}+\sqrt{2}\, \left(c^{2}-2 q_{0} \right)}{8 \beta \sqrt{q_{0}}};\frac{3}{2};\frac{\sqrt{2}\, \sqrt{q_{0}}}{\beta}\right)\right.\notag\\
		&\left.+3 M \left(\frac{14 \beta \sqrt{q_{0}}+\sqrt{2}\, \left(c^{2}-2 q_{0} \right)}{8 \beta \sqrt{q_{0}}};\frac{5}{2};\frac{\sqrt{2}\, \sqrt{q_{0}}}{\beta}\right) \left(\beta \sqrt{q_{0}}\, \sqrt{2}+\frac{c^{2}}{3}-\frac{2 q_{0}}{3}\right)\right]^{-1},
	\end{align}
	
	\begin{align}\label{eq-implicit beta1}
		T(-\ell')-T_{1}(c)=&\left[3 \beta \sqrt{c^{2}-2q_{0}}\, \left(-\sqrt{2}\, \sqrt{q_{0}}+\beta -c +T_{i} \right)\right.\notag\\
		&M\left(\frac{6 \beta \sqrt{q_{0}}+\sqrt{2}\, \left(c^{2}-2 q_{0} \right)}{8 \beta \sqrt{q_{0}}};\frac{3}{2};\frac{\sqrt{2}\, \sqrt{q_{0}}}{\beta}\right)+3 \left(\beta \sqrt{q_{0}}\, \sqrt{2}+\frac{c^{2}}{3}-\frac{2 q_{0}}{3}\right)\notag\\ 
		&M \left(\frac{14 \beta \sqrt{q_{0}}+\sqrt{2}\, \left(c^{2}-2 q_{0} \right)}{8 \beta \sqrt{q_{0}}};\frac{5}{2};\frac{\sqrt{2}\, \sqrt{q_{0}}}{\beta}\right)\sqrt{c^{2}-2 q_{0}}\,+3 \beta \left(-\beta\left(\sqrt{2}\, \sqrt{q_{0}}\right.\right.\notag\\
		&\left.\left.+c -T_{i} \right)  M \left(\frac{2 \beta \sqrt{q_{0}}+\sqrt{2}\, \left(c^{2}-2 q_{0} \right)}{8 \beta \sqrt{q_{0}}};\frac{1}{2};\frac{\sqrt{2}\, \sqrt{q_{0}}}{\beta}\right)\right.\notag\\
		&\left.\left.+M\left(\frac{10 \beta \sqrt{q_{0}}+\sqrt{2}\, \left(c^{2}-2 q_{0} \right)}{8 \beta \sqrt{q_{0}}};\frac{3}{2};\frac{\sqrt{2}\, \sqrt{q_{0}}}{\beta}\right) \left(\beta \sqrt{q_{0}}\, \sqrt{2}+c^{2}-2 q_{0} \right)\right)\right]\notag\\
		&\left[3 \beta \left(-\sqrt{2}\, \sqrt{q_{0}}+\beta -c +T_{i} \right)  M \left(\frac{6 \beta \sqrt{q_{0}}+\sqrt{2}\, \left(c^{2}-2 q_{0} \right)}{8 \beta \sqrt{q_{0}}};\frac{3}{2};\frac{\sqrt{2}\, \sqrt{q_{0}}}{\beta}\right)\right.\notag\\
		&\left.+3 M \left(\frac{14 \beta \sqrt{q_{0}}+\sqrt{2}\, \left(c^{2}-2 q_{0} \right)}{8 \beta \sqrt{q_{0}}};\frac{5}{2};\frac{\sqrt{2}\, \sqrt{q_{0}}}{\beta}\right) \left(\beta \sqrt{q_{0}}\, \sqrt{2}+\frac{c^{2}}{3}-\frac{2 q_{0}}{3}\right)\right]^{-1}.
	\end{align}

	\section{Numerical shooting method for the Arrhenius function}\label{AppendixB}
	In this appendix, we give a discrete method that can be combined with the qualitative results to simulate the bifurcation diagram for the Arrhenius law. To facilitate the use of the Runge-Kutta method for vector filed $-\tilde{X}^{U}$, we take the initial condition $(U(0),T(0))=(0,T_{0}(c))$ or $(U(0),T(0))=(0,T_{1}(c))$, the final condition $((U(\ell'), T(\ell')) = (\frac{1}{1-\alpha}, T_{i})$. In the previous section, where we showed that for the fixed $c_{1}\in[c^{cj},+\infty]$, the transition map $Z (T_0(c_{1}),\beta,c_{1})$ is a monotone continuously increasing function with respect to $\beta$ (see Proposition \ref{prop-rotating}), furthermore, as $\beta \to +\infty$, $Z (T_0(c_{1}),\beta,c_{1}) \to +\infty$; $Z (T_0(c_{1}),\beta,c_{1})<1$ when $\beta>0$ is small enough (see Proposition \ref{prop-betainfty} and Proposition \ref{prop-beta0}). Consequently, by continuity and monotonicity, there exists a unique value $\beta=\beta_{0}(c_{1})\in (0,+\infty)$ such that $\mathcal{Z} (T_0(c_{1}),\beta,c_{1})=1$.
	The corresponding $\ell'$ of each pair of parameters $(\beta,c)$ is finite, has the estimate
	\begin{align}
		\frac{1}{\beta\exp(-\frac{T_{a}}{2T_{0}(c)})}\leqslant\ell'\leqslant\frac{1}{\beta\exp(-\frac{T_{a}}{T_{i}})}.
	\end{align}
	To illustrate the numerical method, we consider the computation of the curve $\beta_0(c)$. The procedure involves discretizing the parameter space and determining $\beta_0(c)$ for each fixed $c$. The steps are as follows:
	
	\begin{enumerate}
		\item 
		We restrict the computational region to $(\beta,c)\in[\beta^{low},\beta^{up}]\times[c^{cj},c^{up}]$, with $c^{up}>c^{cj}$, $\beta^{up}>\beta^{low}>0$. 
		
		\item 
		We discretize the parameter space with uniform meshes: the interval $[c^{cj},c^{up}]$ with step size $h_c>0$ and $[\beta^{low},\beta^{up}]$ with step size $h_\beta>0$. The resulting grid points are given by:
		\begin{align}
			(\beta_j,c_i)=(\beta^{low}+jh_\beta,c^{cj}+ih_c),
		\end{align}
		where $i=0,1,\ldots,N_c$, $j=0,1,\ldots,N_\beta$, with $N_c=\frac{c^{up}-c^{cj}}{h_c}$ and $N_\beta=\frac{\beta^{up}-\beta^{low}}{h_\beta}$.
		
		%Discretizing the interval $c\in[c^{cj},c^{up}]$ with a step $h_{c}>0$ and the interval $\beta\in[\beta^{low},\beta^{up}]$ with step $h_{\beta}$. All points on the grid can be expressed as $(c_{i},\beta_{j})=(c^{cj}+ih_{c},\beta^{low}+jh_{\beta}),\ i=0,1...\frac{c^{up}-c^{cj}}{h_{c}},\ j=0,1...\frac{\beta^{up}-\beta^{low}}{h_{\beta}}$;
		\item 
		%Fix $c_{i}$ first, for all the value of $\beta_{j}$, compute the solution of the system $\eqref{eq-tildeXU}$ by the 4/5th order Runge-Kutta  method on the interval $A=[0,\frac{1}{K\beta\exp(-\frac{T_{a}}{T_{i}})}]$ with the initial condition $(T(0)=T_{0}(c_{i}),U(0)=0)$, the step $d_{(c_{i},\beta_{j})}$ is the step of the numerical method on the interval A. The discretization of interval A can be represented as $\tau_{k}=kd_{(c_{i},\beta_{j})},\ k=0,1...\frac{1}{Kd_{(c_{i},\beta_{j})}\beta\exp(-\frac{T_{a}}{T_{i}})}$; %RelTol=1e-19,AbsTol=1e-22;
		For each fixed $c_i$, we solve system $-\tilde{X}^{U}$ for all values of $\beta_j$ using the 4/5th order Runge-Kutta method. The integration is performed over $[0,\frac{1}{K\beta\exp(-\frac{T_a}{T_i})}]$, with initial conditions $(T(0),U(0))=(T_0(c_i),0)$. The numerical step size $d_{(\beta_j,c_i)}$, the temporal mesh points are given by:
		\begin{align}
			\tau_k=kd_{(\beta_j,c_i)}, \quad k=0,1,\ldots,N_{ij},
		\end{align}
		where $N_{ij}=\frac{1}{Kd_{(\beta_j,c_i)}\beta\exp(-\frac{T_a}{T_i})}$.
		\item
		For each fixed $c_i$ and $\beta_{j}$, we take $\bar{\ell}'{(\beta_j,c_i)}=\tau_{s}>0$ as an approximation of $\ell'(\beta,c)$ such that:
		\begin{align}
			|U(\bar{\ell}'{(\beta_j,c_i)},\beta_j,c_i)-\frac{1}{1-\alpha}| =|U(\tau_s,\beta_j,c_i)-\frac{1}{1-\alpha}|=\underset{\tau_k}{\min} |U(\tau_k,\beta_j,c_i)-\frac{1}{1-\alpha}|.
		\end{align}
		%	This value $\bar{\ell}'{(c_i,\beta_j)}$ serves as our numerical approximation of $\ell'$.
		\item 
		%Find $\bar{\beta}_{0}(c_{i})>0$ such that $   |T(\bar{\ell}'_{(c_{i},\bar{\beta}_{0}(c_{i}))},c_{i},\bar{\beta}_{0}(c_{i}))-T_{i}|=\underset{\beta_{j}}{\min}|T(\bar{\ell}'_{(c_{i},\beta_{j})},c_{i},\beta_{j})-T_{i}|$;
		
		For each $c_i$, we determine $\beta_{0}(c_i)=\beta_{r}>0$ such that:
		\begin{align}
			|T(\bar{\ell}'{(\beta_0(c_i),c_i)},\beta_0(c_i),c_i)-T_i| =|T(\bar{\ell}'{(\beta_r,c_i)},\beta_r,c_i)-T_i|= \underset{\beta_{j}}{\min}|T(\bar{\ell}'{(\beta_j,c_i)},\beta_j,c_i)-T_i|.
		\end{align} 
		\item 
		The bifurcation curve $\beta_0(c)$ is visualized by plotting the discrete data points $(\beta_{0}(c_i),c_i)$, where $i=0,1,\ldots,N_c$ with $N_c=\frac{c^{up}-c^{cj}}{h_c}$ in the parameter space.
	\end{enumerate}
	
	Following the numerical procedure described for $\beta_0(c)$, we can similarly obtain the bifurcation curve $\beta_1(c)$ by replacing $(U(0),T(0))=(T_1(c),0)$ with the $(U(0),T(0))=(T_0(c),0)$.

	The numerical computations are performed with the following specifications:
	\begin{itemize}
		\item 
		MATLAB's ODE solver $\mathtt{ODE45}$ with relative tolerance $\mathtt{RelTol}=10^{-19}$ and absolute tolerance $\mathtt{AbsTol}=10^{-22}$;
		\item
		Discretization steps: $h_c=0.02$, $h_\beta=0.006$;
		\item
		Parameter bounds: $[\beta^{low},\beta^{up}]\times[c^{cj},c^{up}]=[0.01,9]\times[0,9]$.
	\end{itemize}
	Here, we remark that the integer index $N_{ij}$ is well defined because the adaptive time-stepping algorithm dynamically adjusts $d_{(c_i,\beta_j)}$ to meet the specified error tolerances while preserving the discretization of the fixed time domain.
	
	Finally, note that to obtain a numerical simulation of the bifurcation curves $\beta_0(c)$ and $\beta_1(c)$ in suitable area, it may be necessary to adjust the parameter ranges $[c^{cj},c^{up}]$ and $[\beta^{low},\beta^{up}]$ according to the specific behavior of the system.

	\end{document}